\documentclass[leqno,12pt]{amsart}
\usepackage{setspace}
\usepackage[english]{babel}
\usepackage{bussproofs}
\usepackage{csquotes}
\usepackage{dirtytalk}
\usepackage{mathtools}
\usepackage{mathrsfs}
\usepackage{latexsym}
\usepackage{enumitem}
\usepackage{amsthm}
\usepackage{amssymb}
\DeclareMathAlphabet{\mathpzc}{OT1}{pzc}{m}{it}
\usepackage[latin1]{inputenc}
\usepackage{afterpage}

\usepackage{bbm}
\usepackage[rgb,dvipsnames]{xcolor}
\usepackage{tikz} 
\usetikzlibrary{decorations.text}
\usetikzlibrary{arrows,arrows.meta,hobby}
\usetikzlibrary{shapes.misc, positioning}

\usepackage{tikz-cd}

\usetikzlibrary{cd}

\newtheorem{theorem}{Theorem}[section]
\newtheorem{maintheorem}{Main Theorem}[section]
\newtheorem{conjecture}{Conjecture}[section]
\newtheorem{lemma}[theorem]{Lemma}
\newtheorem{proposition}[theorem]{Proposition}
\newtheorem{corollary}[theorem]{Corollary}

\newtheorem{fact}[theorem]{Fact}
\newtheorem{claim}[theorem]{Claim}
\theoremstyle{definition}
\newtheorem{definition}[theorem]{Definition}
\newtheorem{example}[theorem]{Example}
\theoremstyle{remark}

\newtheorem{question}{Question}

\def\hook{\upharpoonright}
\def\forces{\Vdash}

\def\Me{\mathcal M}
\def\Null{\mathcal N}
\def\ZFC{\mathsf{ZFC}}

\def\Kb{\mathcal K}
\def\baire{\omega^\omega}
\def\bbb{(\omega^\omega)^{\omega^\omega}}

\def\mfc{\mathfrak{c}}
\def\mfb{\mathfrak b}
\def \mfd{\mathfrak{d}}
\def\GCH {\mathsf{GCH}}
\def\CH {\mathsf{CH}}
\def\Q{\mathbb Q}
\def\P{\mathbb P}
\def\D{\mathbb D}

\def\cc{\mathfrak{c}}
\def\bb{\mathfrak{b}}
\def\omom{\omega^\omega}
\def\M{\mathcal M}

\def\cof{{\rm cof}}

\usepackage[margin=1.3in]{geometry}

\begin{document}

\title{Higher Dimensional Cardinal Characteristics for Sets of Functions II}

\author[Brendle]{J\"org Brendle}
\address[J. ~Brendle]{Graduate School of System Informatics, Kobe University, Rokko-Dai 1-1, Nada-Ku, Kobe 657-8501, JAPAN}
\email{brendle@kobe-u.ac.jp}

\author[Switzer]{Corey Bacal Switzer}
\address[C.~B.~Switzer]{Institut f\"{u}r Mathematik, Kurt G\"odel Research Center, Universit\"{a}t Wien, Kolingasse 14-16, 1090 Wien, AUSTRIA}
\email{corey.bacal.switzer@univie.ac.at}

\thanks{\emph{Acknowledgements:} The first author is partially supported by Grant-in-Aid for Scientific Research (C) 18K03398, Japan
Society for the Promotion of Science. The second author would like to thank the
Austrian Science Fund (FWF) for the generous support through grant number Y1012-N35.}

\date{}

\maketitle

\begin{abstract}
We study the values of the higher dimensional cardinal characteristics for sets of functions $f:\baire \to \baire$ introduced by the second author in \cite{Switz19}. We prove that while the bounding numbers for these cardinals can be strictly less than the continuum, the dominating numbers cannot. We compute the bounding numbers for the higher dimensional relations in many well known models of $\neg\CH$ such as the Cohen, random and Sacks models and, as a byproduct show that, with possibly one exception, for the bounding numbers there are no $\ZFC$ relations between them beyond those in the higher dimensional Cicho\'{n} diagram. In the case of the dominating numbers we show that in fact they collapse in the sense that modding out by the ideal does not change their values. Moreover, they are closely related to the dominating numbers $\mfd^\lambda_\kappa$.
\end{abstract}

\section{Introduction}
In \cite{Switz19} the second author introduced 18 cardinal characteristics on the set of functions $f:\baire \to \baire$ generalizing standard cardinal characteristics on $\omega$ by replacing relations on $\omega$ such as $\leq$ by their ``mod finite" counterparts e.g. $\leq^*$ and the ``mod finite" quotienting by ``mod $\mathcal I$" where $\mathcal I$ is some well studied $\sigma$-ideal on the reals such as the ideal of meager or measure zero sets. In that article it was shown that these cardinals can be organized into two diagrams of $\ZFC$-provable relations similar to the Cicho\'n diagram (see below Figures \ref{figurenull},\ref{figuremeager}). Several relations between these ``higher dimensional" cardinals and their brethren on $\omega$ were established as well as a number of consistency results, albeit all in the context where $\CH$ holds. 

In this paper we look at when $\CH$ fails, a situation that turns out to be infinitely more flexible and interesting. Working in this context we prove, for the bounding numbers of the higher dimensional relations, that the provable inequalities from the higher Cicho\'{n} diagrams are in fact the only ones, with possibly one exception which we discuss. The proof consists of considering how standard iterated forcing models of the reals (Cohen, Random, Sacks, etc) change the higher dimensional cardinals. We also present a result, due to the first author, that shows that, in contrast to what appeared in \cite{Switz19}, the $\CH$ context was not so interesting after all. We then study the dominating numbers when $\CH$ fails and establish many consistent inequalities while also exhibiting that several surprising $\ZFC$ relations hold between them that fail in the dual case. The result shows that these higher dimensional cardinals fail in a strong sense to satisfy the type of duality results that often characterize the classical cardinal characteristics on $\omega$ and even their higher analogues on arbitrary, regular $\kappa$. To state these results more precisely we recall some definitions.

Given a set $X$ and a binary relation $R$ on $X$ we say that a set $A \subseteq X$ is $R$-{\em bounded} if there is a single $x \in X$ so that $yRx$ for every $y \in A$. We say that $A$ is $R$-{\em unbounded} if it is not $R$-bounded. We say that $A \subseteq X$ is $R$-{\em dominating} if for every $x \in X$ there is a $y \in A$ so that $xRy$. For any such $R$ and $X$ we denote by $\mfb(R)$ the {\em bounding} number of $R$, i.e. the least size of an $R$-unbounded set and by $\mfd(R)$ the {\em dominating} number of $R$, i.e. the least size of an $R$-dominating set.

We work primarily in Baire space. Recall that a {\em slalom} is a function $s:\omega \to [\omega]^{<\omega}$ so that for all $n$ we have $|s(n)| \leq n$. Let $\mathcal S$ be the space of slaloms, which we treat as homeomorphic to $\baire$ by any reasonable homeomorphism. We recall the following three relations.

\begin{definition}
Let $x, y \in \baire$ and $s$ be a slalom.
\begin{enumerate}
\item
We say that $g$ {\em eventually dominates} $f$, in symbols $f \leq^* g$ if for all but finitely many $k < \omega$ we have $f(k) \leq g(k)$. The bounding and dominating numbers for this relation are the classical cardinals $\mfb$ and $\mfd$, see \cite[Section 2]{BlassHB}.
\item
We say that $g$ is {\em eventually different from} $f$, in symbols $f \neq^* g$ if for all but finitely many $k < \omega$ we have $f(k) \neq g(k)$. A theorem of Miller states that $\mfb(\neq^*) = {\rm non}(\Me)$ and $\mfd(\neq^*) = {\rm cov}(\Me)$ where $\Me$ is the ideal of meager sets, see \cite[Theorem 5.9]{BlassHB}.
\item
We say that $f$ is {\em eventually captured by} $s$, in symbols $f \in^* s$ if for all but finitely many $k < \omega$ we have $f(k) \in s(k)$. By a theorem of Bartoszy\'{n}ski we have $\mfb(\in^*) = {\rm add}(\Null)$ and $\mfd(\in^*) = {\rm cov}(\Null)$ where $\Null$ is the ideal of Lebesgue measure zero sets, see \cite[Theorem 5.14]{BlassHB}.
\end{enumerate}
\end{definition}

Throughout let $\mathcal I$ range over the ideal of Lebesgue measure zero sets, denoted $\Null$, the ideal of meager sets, denoted $\Me$, and the ideal of $\sigma$-compact subsets of $\baire$, denoted $\Kb$. Note that by a well known result of Rothberger, the latter ideal can also be characterized as the ideal generated by $\leq^*$-bounded sets, see \cite[Theorem 2.8]{BlassHB}. Let $R$ range over $\leq^*$, $\neq^*$ and $\in^*$. Denote by $\bbb$ the space of functions $f:\baire \to \baire$. The following definition is the main object of study in the paper. 

\begin{definition}
Let $\mathcal I \in \{\mathcal N, \mathcal M, \mathcal K\}$ and $R \in \{\leq^*, \neq^*, \in^*\}$. Let $f, g \in \bbb$, or in the case of $R = \in^*$, let $g: \baire \to \mathcal S$. We define the relation $R_\mathcal I$ by $f R_\mathcal I g$ if and only if $\{x \in \baire \; | \; \neg (f(x) R g(x))\} \in \mathcal I$. In other words, $g$ is an $R$-bound for $f$ on an $\mathcal I$-measure one set.
\end{definition}

By varying $\mathcal I$ and $R$ this definition gives nine new relations and 18 new cardinal characteristics, a bounding and dominating number for each. For readability, let us give the details below for the case of the null ideal. Similar statements hold for $\Me$ and $\Kb$. First let's see explicitly what each relation $R_\mathcal I$ is. On the two lists below let $f, g:\baire \to \baire$ and $h:\baire \to \mathcal S$.

\begin{enumerate}
\item
$f \neq^*_\Null g$ if and only if for all but a measure zero set of $x \in \baire$ we have that $f (x) \neq^* g(x)$.

\item
$f \leq^*_\Null g$ if and only if for all but a measure zero set of $x \in \baire$ we have that $f(x) \leq^* g(x)$.

\item
$f \in^*_\Null h$ if and only if for all but a measure zero set of $x \in \baire$ we have that $f(x) \in^* h(x)$.
\end{enumerate}

For the cardinals now we get the following. Note that $\neg x \neq^* y$ means $\exists^\infty n \, x(n) = y(n)$ and the same for the other relations.

\begin{enumerate}

\item
$\mfb(\neq^*_\Null)$ is the least size of a $\neq^*_\Null$-unbounded set $A \subseteq \bbb$ i.e. $A$ is such that for each $f:\baire \to \baire$ there is a $g \in A$ so that the set of $\{x \; | \; \exists^\infty n \, g(x)(n) = f(x)(n)\}$ is not measure zero.

\item
$\mfd(\neq^*_\Null)$ is the least size of a $\neq^*_\Null$-dominating set $A \subseteq \bbb$ i.e. $A$ is such that for every $f:\baire \to \baire$ there is a $g \in A$ so that $\mu (\{x \; | \; f(x)\neq^* g(x)\} )= 1$.

\item
$\mfb(\leq^*_\Null)$ is the least size of a $\leq^*_\Null$-unbounded set $A \subseteq \bbb$ i.e. $A$ is such that for each $f:\baire \to \baire$ there is a $g \in A$ so that the set of $\{x \; | \; \exists^\infty n \, f(x)(n) < g(x)(n)\}$ is not measure zero.

\item
$\mfd(\leq^*_\Null)$ is the least size of a $\leq^*_\Null$-dominating set $A \subseteq \bbb$ i.e. $A$ is such that for every $f:\baire \to \baire$ there is a $g \in A$ so that $\mu (\{x \; | \; f(x)\leq^* g(x)\}) = 1$.

\item
$\mfb(\in^*_\Null)$ is the least size of a $\in^*_\Null$-unbounded set $A \subseteq \bbb$ i.e. $A$ is such that for each $f:\baire \to \mathcal S$ there is a $g \in A$ so that the set of $\{x \; | \; \exists^\infty n \, g(x)(n) \notin f(x)(n)\}$.

\item
$\mfd(\in^*_\Null)$ is the least size of a $\leq^*_\Null$-dominating set $A \subseteq \bbb$ i.e. $A$ is such that for every $f:\baire \to \baire$ there is a $g \in A$ so that $\mu (\{x \; | \; f(x)\in^* g(x)\}) = 1$.

\end{enumerate}

Let us also recall provable relations for these cardinals. The following diagrams are \say{higher dimensional} analogues of Cicho\'n's diagram.
\begin{theorem}[Theorem 1.1, \cite{Switz19}]
Interpreting $\to$ as $\leq$ the inequalities shown in Figures 1 and 2 are all provable in $\ZFC$.

\begin{figure}[h]
\centering
  \begin{tikzpicture}[scale=1.5,xscale=2]
     \draw 
           (1,0) node (Bin*) {$\mfb(\in_\Null^*)$}
           (1,1) node (Bleq*) {$\mfb(\leq_\Null^*)$}
           (1,2) node (Bneq*) {$\mfb(\neq_\Null^*)$}
           (2,0) node (Dneq*) {$\mfd(\neq_\Null^*)$}
           (2,1) node (Dleq*) {$\mfd(\leq_\Null^*)$}
           (2,2) node (Din*) {$\mfd(\in_\Null^*)$}
           ;
     \draw[->,>=stealth]
            (Bin*) edge (Bleq*)
            (Bleq*) edge (Bneq*)
            (Bin*) edge (Dneq*)
            (Bleq*) edge (Dleq*)
            (Bneq*) edge (Din*)
            (Dneq*) edge (Dleq*)
            (Dleq*) edge (Din*)
;      
  \end{tikzpicture}
\caption{Higher Dimensional Cardinal Characteristics Mod the Null Ideal}
\label{figurenull}
\end{figure}

\begin{figure}[h]
\centering
  \begin{tikzpicture}[scale=1.5,xscale=2]
     \draw 
           (1,0) node (Bin*) {$\mfb(\in_\Me^*)$}
           (1,1) node (Bleq*) {$\mfb(\leq_\Me^*)$}
           (1,2) node (Bneq*) {$\mfb(\neq_\Me^*)$}
           (2,0) node (Dneq*) {$\mfd(\neq_\Me^*)$}
           (2,1) node (Dleq*) {$\mfd(\leq_\Me^*)$}
           (2,2) node (Din*) {$\mfd(\in_\Me^*)$}
	     (0,0) node (BinK*) {$\mfb(\in_\Kb^*)$}
           (0,1) node (BleqK*) {$\mfb(\leq_\Kb^*)$}
           (0,2) node (BneqK*) {$\mfb(\neq_\Kb^*)$}
           (3,0) node (DneqK*) {$\mfd(\neq_\Kb^*)$}
           (3,1) node (DleqK*) {$\mfd(\leq_\Kb^*)$}
           (3,2) node (DinK*) {$\mfd(\in_\Kb^*)$}
           ;
     \draw[->,>=stealth]
            (Bin*) edge (Bleq*)
            (Bleq*) edge (Bneq*)
            (Bin*) edge (Dneq*)
            (Bleq*) edge (Dleq*)
            (Bneq*) edge (Din*)
            (Dneq*) edge (Dleq*)
            (Dleq*) edge (Din*)
            (Dleq*) edge (DleqK*)
            (Dneq*) edge (DneqK*)
            (Din*) edge (DinK*)
            (BleqK*) edge (Bleq*)
            (BneqK*) edge (Bneq*)
            (BinK*) edge (Bin*)
            (BinK*) edge (BleqK*)
            (BleqK*) edge (BneqK*)
            (DneqK*) edge (DleqK*)
	      (DleqK*) edge (DinK*)

;
 \end{tikzpicture}
\caption{Higher Dimensional Cardinal Characteristics Mod the Meager and $\sigma$-Compact Ideals}
\label{figuremeager}
\end{figure}

\label{cichondiagramtheorem}

\end{theorem}
\newpage
We also get the following bounds on these cardinals in terms of $\omega$.
\begin{lemma}[Proposition 3.1 of \cite{Switz19}]
Let $R \in \{\in^*, \leq^*, \neq^*\}$ and $\mathcal I \in \{\Kb, \Null, \Me\}$. In $\ZFC$ it's provable that $\mfb(R) \leq \mfb(R_\mathcal I) \leq \mfb(R)^{{\rm non}(\mathcal I)}$.
\label{omegabounds}
\end{lemma}

It was previously left open whether these cardinals could be less than the continuum. We answer this in the affirmative for the $\mfb(R_\mathcal I)$ cardinals and in the negative for the $\mfd (R_\mathcal I)$ cardinals. In the case of the $\mfb(R_\mathcal I)$ cardinals this is a consequence of the following theorem.

\begin{maintheorem}
Suppose $R, S  \in \{\in^*, \leq^*, \neq^*\}$ and $\mathcal I, \mathcal J \in \{\Kb, \Null, \Me\}$. With possibly one exception, if $\mathfrak{b}(R_\mathcal I) \leq \mathfrak{b}(S_\mathcal J)$ is not an inequality in either of Figures \ref{figurenull}, \ref{figuremeager} then it is consistently false.  Moreover, for every $R$ and $\mathcal I$, $\mfb(R_\mathcal I)$ can be strictly less than the continuum. More succinctly, with one possible exception, there are no $\ZFC$ provable inequalities between the cardinals $\mfb(R_\mathcal I)$ beyond what is stated in Theorem \ref{cichondiagramtheorem} and Lemma \ref{omegabounds}.
\label{mainthm1}
\end{maintheorem}

The one possible exception is that it remains open whether $\mfb(\neq^*_\Null) < \mfb(\in^*_\Kb)$ is consistent. This is discussed as Question \ref{conineq} below.

Main Theorem \ref{mainthm1} follows cumulatively from the analysis of the various models in Section 4 below. The second main theorem is the following, due to the first author.

\begin{maintheorem}
Assume $\CH$. For all $R \in \{\in^*, \leq^*, \neq^*\}$ and all $\mathcal I \in \{\Kb, \Me, \Null\}$ we have $\mfb(R_\mathcal I) = \mfb_{\aleph_1}$ and $\mfd(R_\mathcal I) = \mfd_{\aleph_1}$.
\label{CHtheorem}
\end{maintheorem}

Here, $\mfb_\kappa$ and $\mfd_\kappa$ are the generalized bounding and dominating numbers for the eventual domination relation on a given cardinal $\kappa$. Main Theorem \ref{CHtheorem} addresses a number of questions left open about the models constructed in Section 4 of \cite{Switz19}. There, consistent inequalities between the cardinals in Figures \ref{figurenull} and \ref{figuremeager} were constructed under $\CH$. From this theorem we see that, in fact, there was an easier way to obtain those results.

We also study the cardinals $\mfd(R_\mathcal I)$ when $\CH$ fails and show that the above Main Theorem \ref{CHtheorem} can fail in this context. 

\begin{maintheorem}
For all $\mathcal I \in \{\Null, \Me, \Kb\}$ it is consistent that $\mfd(\leq^*_\mathcal I)< \mfd(\in^*_\mathcal I)$ and $\mfd (\neq^*_\mathcal I)< \mfd(\leq^*_\mathcal I)$.
\label{mainthm3}
\end{maintheorem}

However, surprisingly, the ideal does not matter for the $\mfd(R_\mathcal I)$ cardinals (in contrast to the $\mfb(R_\mathcal I)$ cardinals), a result also due to the first author.

\begin{maintheorem}
For all $R \in \{\in^*, \leq^*, \neq^*\}$ $\ZFC$ proves that $\mfd(R_\Null) = \mfd(R_\Me) = \mfd (R_\Kb)$.
\label{mainthm4}
\end{maintheorem}

Main Theorems \ref{mainthm1}, \ref{mainthm3} and \ref{mainthm4} above taken together show that duality fails strongly for these higher cardinal characteristics, making their theory starkly different from that of the classical cardinal characteristics as studied for example in \cite{BarJu95} or \cite{BlassHB} or even their higher dimensional analogues on $\kappa$, see \cite{cichonunctble}. See \cite[Section 4]{BlassHB} for a particular discussion of this point in the context of cardinals on $\omega$.

The rest of this article is organized as follows. In the next section we study the $\CH$ case, culminating in the proof of Main Theorem \ref{CHtheorem}. In Section 3 we prove a number of $\ZFC$ results concerning various implications between the higher dimensional bounding numbers, cardinals on $\omega$ and the continuum. In Section 4 we study the bounding numbers for the relations $R_\mathcal I$ in well known models of $\neg \CH$, culminating in a proof of Main Theorem \ref{mainthm1}. In Section 5 we study the $\mathfrak{d}(R_\mathcal I)$ cardinals, showing Main Theorem \ref{mainthm4}. In Section 6 we study consistent inequalities between the dominating numbers culminating in the proof of Main Theorem \ref{mainthm3}. Section 7 concludes with some open questions and final remarks.

\section{The $\CH$ Case and Eventual Domination in Generalized Baire Space}
In this section we prove Main Theorem \ref{CHtheorem}. In fact we show something slightly stronger, from which Main Theorem \ref{CHtheorem} follows immediately. We show the following.

\begin{theorem}
Assume $\mfc = \mfc^{{<}\mfc}$ is a successor cardinal. Fix $R \in \{\in^*, \leq^*, \neq^*\}$ and $\mathcal I \in \{\Kb, \Me, \Null\}$. If $\mfb(R) = {\rm add}(\mathcal I) = \mfc$ then $\mfb(R_\mathcal I) = \mfb_\mfc$ and $\mfd(R_\mathcal I) = \mfd_\mfc$. In particular these equalities hold under $\CH$.
\label{CHtheorem2}
\end{theorem}

Again recall that here, for any $\kappa$, that if $f, g \in \kappa^\kappa$ then $f\leq^* g$ if $\{\alpha \; | \; f(\alpha) \nleq g(\alpha)\}$ is bounded. The cardinals $\mfb_\kappa$ and $\mfd_\kappa$ denote the bounding and dominating numbers for this relation respectively.

Of course, in light of Main Theorem \ref{mainthm4} adding in the ideal $\mathcal I$ for the dominating numbers is redundant, however we are interested in a duality result for the bounding numbers as well so we leave it in. Towards proving Theorem \ref{CHtheorem2} we start with the following Lemma. 

\begin{lemma}
Fix $R \in \{\in^*, \leq^*, \neq^*\}$ and $\mathcal I \in \{\Kb, \Me, \Null\}$. If $\mfb(R) = {\rm non}(\mathcal I) = \mfc$ then $\mfd(R_\mathcal I) \leq \mfd_\mfc$ and $\mfb_\mfc \leq \mfb(R_\mathcal I)$.
\label{upperboundD}
\end{lemma}

\begin{proof}
Fix $R$ and $\mathcal I$ as in the statement of the lemma and assume $\mfb(R) = {\rm non}(\mathcal I) = \mfc$. Enumerate $\baire$ as $\{x_\alpha \; | \; \alpha < \mfc\}$ and for each $\alpha < \mfc$ fix an element $y_\alpha$ of $\baire$  (or $\mathcal S$ in the case that $R = \in^*$) so that for all $\beta \leq \alpha$, $x_\beta R y_\alpha$. Such a $y_\alpha$ exists by the assumption on $\mfb(R)$. Define $\Psi:\mfc^\mfc \to \bbb$ by letting $\Psi(f):\baire \to \baire$ be given by $\Psi(f)(x_\alpha) = y_{f(\alpha)}$. Let $D = \{f_\gamma \; | \; \gamma < \mfd_\mfc\} \subseteq \mfc^\mfc$ be a dominating family of minimal size. Now, let $\Psi(D) = \{\Psi(f_\gamma) \; |\; \gamma < \mfd_\mfc\}$. We claim that this is a $R_\mathcal I$ dominating family. To see this, fix $g:\baire \to \baire$ and consider $\hat{g}:\mfc \to \mfc$ given by $\hat{g}(\alpha) = \beta$ if and only if $g(x_\alpha) = x_\beta$. Now there is some $f_\gamma \in D$ which dominates $\hat{g}$ mod $<{\mfc}$. But then the set of reals $x \in \baire$ for which $\neg g(x) R \Psi(f_\gamma)(x)$ has size less than $\mfc$ and hence, by the assumption on ${\rm non}(\mathcal I)$ we get that $g R_\mathcal I \Psi(f_\gamma)$ as needed.

The other inequality is essentially proved by duality. More precisely, suppose $\kappa < \mfb_\mfc$ and $A \subseteq \bbb$ is a set of functions of size $\kappa$. Let $\hat{A} = \{\hat{g} \; | \; g \in A\}$ where $g \mapsto \hat{g}$ is as defined in the previous paragraph. Since $\hat{A}$ has size $\kappa$ it's bounded, say by some $f \in \mfc^\mfc$. Essentially the same proof as in the previous paragraph then shows $\Psi(f)$ is an $R_\mathcal I$-bound on $A$.
\end{proof}

\begin{lemma}
Fix $R \in \{\in^*, \leq^*, \neq^*\}$ and $\mathcal I \in \{\Kb, \Me, \Null\}$. If ${\rm add}(\mathcal I) = \mfc$ then $\mfd_\mfc (\neq^*) \leq \mfd(R_\mathcal I)$ and $\mfb(R_\mathcal I) \leq \mfb_\mfc(\neq^*)$.
\label{lowerboundD}
\end{lemma}

Here recall that if $f, g \in \kappa^\kappa$ then $f\neq^* g$ if the set of $\alpha$ for which $f(\alpha) = g(\alpha)$ has size less than $\kappa$. The bounding and dominating numbers for this relation are denoted respectively $\mfb_\kappa(\neq^*)$ and $\mfd_\kappa(\neq^*)$, see \cite[Definition 16]{cichonunctble}.

\begin{proof}
Fix $R$ and $\mathcal I$ as in the statement of the lemma and assume ${\rm add}(\mathcal I) = \mfc$. We show that $\mfd_\mfc (\neq^*) \leq \mfd(\neq^*_\mathcal I)$. Let $\{A_\gamma \; | \; \gamma < \mfc\}$ be an increasing, cofinal family of sets from $\mathcal I$ (that such a family exists follows from ${\rm add}(\mathcal I) = \mfc$). Choose $\{x_\gamma \; | \; \gamma < \mfc\} \subseteq \baire$ so that for all $\gamma < \mfc$ we have $x_\gamma \notin A_\gamma$. Now fix an $\neq^*_\mathcal I$-dominating family $\mathcal F = \{f_\alpha \; | \; \alpha < \mfd(\neq^*_\mathcal I) \}$. Let $F:\mfc \to \baire$ be a bijection. Define $g_\alpha: \mfc \to \mfc$ by $g_\alpha(\gamma) = F^{-1}(f_\alpha(x_\gamma))$. We claim that $\{g_\alpha\; | \; \alpha < \mfd(\neq^*_\mathcal I)\}$ is a dominating family for the relation $\neq^*$ on $\mfc^\mfc$. Indeed suppose that $g:\mfc \to \mfc$ is arbitrary. Let $f:\baire \to \baire$ be such that $f(x_\gamma) = F(g(\gamma))$ (the other values of $f$ are irrelevant). By assumption we know that for some $\alpha < \mfd(\neq^*_\mathcal I)$ the set $\{x \; | \; \exists^\infty n \, f(x)(n) = f_\alpha(x)(n)\}$ is in ${\mathcal I}$ and so it lies in some $A_\gamma$. As a result for all but ${<}\mfc$ many $\gamma$ we have that $F(g(\gamma)) = f(x_\gamma) \neq^* f_\alpha(x_\gamma) = F(g_\alpha(\gamma))$ and thus $g(\gamma) \neq g_\alpha(\gamma)$ for all but ${<}\mfc$ many $\gamma$ as needed. 

The proof for the bounding numbers is dual to the one above. The details are left to the reader.
\end{proof}

Putting together Lemmas \ref{upperboundD} and \ref{lowerboundD} alongside the fact that if $\kappa$ is a successor cardinal for which $\kappa^{<\kappa} = \kappa$ then $\mfb_\kappa(\neq^*) = \mfb_\kappa$ and $\mfd_\kappa(\neq^*) = \mfd_\kappa$, see \cite[Theorem 21]{cichonunctble}, now implies Theorem \ref{CHtheorem2}.

\section{$\ZFC$ Results for the Bounding Numbers}
We begin by giving a number of $\ZFC$ results about the higher dimensional cardinals of the form $\mfb(R_\mathcal I)$ and their counterparts on $\omega$ that will be used to prove Main Theorem \ref{mainthm1}. The first lemma was originally proved in \cite[Lemma 3.2]{Switz19}. The proof is simply a standard diagonal argument, making use of the assumptions.

\begin{lemma}[Lemma 3.2 of \cite{Switz19}]
For each $\mathcal I \in \{\Null, \Me, \Kb\}$ and $R \in \{ \in^*, \leq^*, \neq^*\}$, if $\mathfrak{b}(R) = {\rm non} (\mathcal I) = \mfc$ then $\mfc^+ \leq \mathfrak{b}(R_\mathcal I)$.
\label{lemma5}
\end{lemma}

This lemma is actually strengthened by Lemma \ref{upperboundD}, however this is the version we will need in Section 4 so we quote it here. The next two lemmas show how relations between the cardinals on $R$ and $\mathcal I$ can have consequences for the value of $\mfb(R_\mathcal I)$.

\begin{lemma}
Suppose $\mfb (R) < {\rm cov}(\mathcal I)$. Then $\mfb (R_\mathcal I) = \mfb (R)$.
\label{lemmab<cov}
\end{lemma}

\begin{proof}
By Lemma \ref{omegabounds} it suffices to prove that under the hypothesis $\mfb(R_\mathcal I) \leq \mfb(R)$. Let $A = \{x_\alpha \; | \; \alpha < \mfb(R)\}$ be an $R$-unbounded family of minimal size and for each $\alpha < \mfb(R)$ let $c_\alpha:\baire \to \baire$ be the constant function with value $x_\alpha$. We claim that the family $\{c_\alpha \; | \; \alpha< \mfb(R)\}$ forms an $R_\mathcal I$-unbounded family. To see this, suppose that $f:\baire \to \baire$ (or $f:\baire \to \mathcal S$) were a bound. Define a function $h \in \bbb$ so that for each $x\in \baire$ let $h(x) \in A$ be $R$-unbounded by $f(x)$. Since no family of sets in $\mathcal I$ of size $\mfb(R)$ is covering, there must be some $x_\alpha \in A$ so that $h^{-1}(\{x_\alpha\})$ is not in $\mathcal I$. But this means that on a non-$\mathcal I$ set $c_\alpha$ is not bounded by $f$, so $f$ is not a bound on $\{c_\alpha \; | \; \alpha < \mfb(R)\}$ as needed.
\end{proof}

\begin{lemma}
If ${\rm cov} (\mathcal I) = \aleph_1$ then $\mfb (\in^*_\mathcal I) \geq \aleph_2$.
\label{lemmacov=1}
\end{lemma}

\begin{proof}
Fix a family $\{A_\alpha \; | \; \alpha < \aleph_1\}$ of sets from $\mathcal I$ covering $\baire$. Without loss, assume that they are disjoint. Let $\{f_\alpha \; | \; \alpha < \aleph_1\}$ be a family of functions $f_\alpha : \baire \to \baire$. We need to provide a bound $g:\baire \to \mathcal S$. We can do this by defining for $x \in A_\alpha$, $g(x)$ to be any slalom bounding the countable set $\{f_\beta (x) \; | \; \beta < \alpha \}$. Now, suppose that $\alpha < \omega_1$ and observe that for all $x \notin \bigcup_{\beta \leq \alpha} A_\beta$ we have that $f_\alpha(x) \in^* g(x)$ and since the latter is in $\mathcal I$ (since it's a countable union), we're done.
\end{proof}

A surprising corollary of this result is the following.
\begin{proposition}
$\mfb (\neq^*_\Null) \geq \aleph_2$.
\label{neqN}
\end{proposition}

\begin{proof}
If ${\rm cov}(\Null) = \aleph_1$ then the result follows from Lemma \ref{lemmacov=1} plus the fact that $\mfb(\in^*_\Null) \leq \mfb(\neq^*_\Null)$. Otherwise ${\rm cov}(\Null) > \aleph_1$. Since ${\rm cov}(\Null) \leq \mfb (\neq^*) = {\rm non}(\Me)$ (in $\ZFC$) we get that $\aleph_2 \leq \mfb(\neq^*) \leq \mfb(\neq^*_\Null)$.
\end{proof}

Oddly enough the other eight numbers of the form $\mfb(R_\mathcal I)$ can all be $\aleph_1$ as we will see. Lemmas \ref{lemma5} and \ref{lemmacov=1} are both actually  special cases of a more general lemma which we prove now.

\begin{lemma}
Fix a cardinal $\kappa$ and assume that there is a family $\mathcal X = \{X_\alpha \; | \; \alpha < \kappa\} \subseteq\mathcal I$ so that for all $\mathcal Y \subseteq \mathcal X$ if $|\mathcal Y| = \mfb(R)$ then $\bigcup \mathcal Y = \baire$ (i.e. $\mathcal Y$ is covering). Then $\mfb(R_\mathcal I) > \kappa$. 
\label{coveringfamilylemma}
\end{lemma}

\begin{proof}
Fix $R$, $\mathcal I$ and $\kappa$ and assume there is a family $\mathcal X = \{X_\alpha \; |\; \alpha < \kappa\}$ as in the statement of the lemma. We need to show that no family of functions $\{f_\alpha \; | \; \alpha < \kappa\} \subseteq \bbb$ is $R_\mathcal I$-unbounded. Fix such a family. We will define a bound on it. First notice that for each $x \in \baire$ the set $\{\gamma \; | \; x \notin X_\gamma\}$ has size ${<}\mfb(R)$ since every family of size $\mfb(R)$ is covering. Now define $g:\baire \to \baire$ so that for all $x \in \baire$ we have $g(x)$ is an $R$-bound on $\{f_\gamma(x) \; | \; x \notin X_\gamma\}$. We claim that his $g$ $R_\mathcal I$-dominates every $f_\alpha$. To see this, fix $\alpha < \kappa$ and consider the set of $x$ so that $\neg ( f_\alpha(x) R g(x) )$. If $x$ is in this set, then by the definition of $g$ we have that $x \in X_\alpha$. Since this later set is in $\mathcal I$ this completes the proof.
\end{proof}

Every computation of a $\mfb(R_\mathcal I)$ cardinal in this paper (and in \cite{Switz19}) factors through one of Lemmas \ref{lemma5}, \ref{lemmacov=1} or \ref{coveringfamilylemma}. Since the later generalizes the former two, every model in which a cardinal of the form $\mfb(R_\mathcal I)$ is greater than $\mfc$ satisfies the hypothesis of Lemma \ref{coveringfamilylemma} with $\kappa$ at least the size of the continuum. We do not know if this is necessary or not. Families of sets $\mathcal X$ as described in Lemma \ref{coveringfamilylemma} are called $(\kappa, \lambda)$-Rothberger families for $\mathcal I$ where $\kappa$ is the cardinality of $\mathcal X$ and $\lambda$ is so that every subfamily of $\mathcal X$ of size $\lambda$ is covering. In this terminology, Lemma \ref{coveringfamilylemma} can be rephrased as saying that the existence of a $(\kappa, \mfb(R))$-Rothberger family for $\mathcal I$ implies that $\mfb(R_\mathcal I) > \kappa$. Such families were investigated in \cite{Cichon89}, see in particular Definition 2.2. In connection with this, note that the case where $\kappa < \mfb(R)$ holds vacuously (both the antecedent and the conclusion) so the interesting case is when $\kappa$ is at least $\mfb(R)$. In this case we obviously have ${\rm cov}(\mathcal I) \leq \mfb(R)$ yet this is not sufficient as, e.g. ${\rm cov}(\Null) = \mfb(\neq^*) = 2^{\aleph_0}$ in the random model but given any family of continuum many null sets we can easily find a subfamily also of size continuum and a random real which escapes all of them. Indeed, we will see in Theorem \ref{random} that $\mfb(\neq^*_\Null) = 2^{\aleph_0}$ so the conclusion of the lemma above is not simply implied by ${\rm cov}(\mathcal I) \leq \mfb(R)$ even. Nevertheless, replacing the null ideal by the meager ideal, such a family does exist in the random model and this is what is used to compute the cardinals $\mfb(\in_\mathcal M)$, $\mfb(\leq^*_\Me)$ and $\mfb(\neq^*_\Me)$ in the random model in Section 4.

\section{Consistency Results for the Bounding Numbers}

In this section we study the values of the bounding numbers in Figures 1 and 2 in various classical models of set theory such as the Cohen, Random and Sacks models. The basic plan is this: we will iterate some well known forcing over a model of $\GCH$ to get a model where $\mfc$ is some specified $\kappa$ and $2^\mfc = \mfc^+$. It will follow from Corollary \ref{lemmamfd} below, which states that $\mfd(R_\mathcal I) > 2^{\aleph_0}$ for all $R$ and $\mathcal I$, that in each model all of the $\mfd (R_\mathcal I)$ cardinals will be $\mfc^+$ so we focus on the $\mfb (R_\mathcal I)$ cardinals. From now on let us fix an arbitrary regular cardinal $\kappa > \aleph_1$. We will need one forcing theoretic lemma.

\begin{lemma}
Assume $\GCH$ and suppose $\mathbb P$ is a finite support or countable support product of proper forcing notions so that if $G \subseteq \mathbb P$ is $V$-generic then in $V[G]$ the reals of $V$ are $\mathcal I$ positive and $R$-unbounded and the set of Borel codes for elements of $\mathcal I$ in $V$ form a Borel basis for $\mathcal I$ in $V[G]$ when reinterpreted. Then the set of functions $f \in \bbb \cap V$ extended arbitrarily to the new reals form an $R_\mathcal I$-unbounded family. In particular, $\forces_\mathbb P \mfb (R_\mathcal I) \leq \aleph_2$.
\label{lemma4}
\end{lemma}

\begin{proof}
Let $\{g_\alpha \; | \; \alpha < \aleph_2\}$ be any set of functions so that the set of restrictions to $V \cap \baire$ is exactly the functions in $V \cap \bbb$. We need to show that this set is unbounded. Fix a function $f:\baire \to \baire$ in $V[G]$. By restricting $f$ to $V \cap \baire$, the fact that the iteration is a product and $\CH$ holds implies that we can find a set of size $\aleph_1$, say $X \subseteq \kappa$ so that $f$ is in fact added by the product restricted to this set, call it $\mathbb P_X$, and this forcing has size $\aleph_1$. Work in $V$ now, and let $\dot{f}$ be the name for $f$. Without loss assume the maximal condition forces $\dot{f}$ is a function from $\baire$ to $\baire$. Enumerate baire space as $\baire = \{x_\alpha \; | \; \alpha < \aleph_1\}$. Also enumerate all pairs $(p_\alpha, B_\alpha)_{\alpha < \omega_1}$ in order type $\omega_1$ of conditions from $\mathbb P_X$ and Borel $\mathcal I$ sets coded in $V$ so that $x_\alpha \notin B_\alpha$. For each $\alpha$ let $q_\alpha \leq p_\alpha$ decide some ground model real not $R$-bounded by $\dot{f} (\check{x}_\alpha)$. Note that such a real exists in the ground model by assumption. Let (in $V$) $h \in \bbb$ be the function which on $x_\alpha$ takes the value $q_\alpha$ decides is unbounded by $\dot{f}$. We claim that for any $\bar{h}$ in $V[G]$ extending $h$ it's not the case that $\bar{h}$ is $R_\mathcal I$-bounded by $\dot{f}$. Indeed, otherwise there is some $\alpha$ so that $p_\alpha \forces \{x \; | \;\neg \check{h}(x) R \dot{f}(x)\} \subseteq B_\alpha$ since the Borel sets in $\mathcal I$ from the ground model form a basis, but $q_\alpha \leq p_\alpha$ and $q_\alpha$ forces that $\neg \check{h}(x_\alpha) R \dot{f}(x_\alpha)$ and $x_\alpha \notin B_\alpha$, which is a contradiction. 
\end{proof}

\subsection{The Cohen Model}
The Cohen model is the model obtained by forcing with the finite support product of $\kappa$ many copies of Cohen forcing over a model of $\GCH$, see \cite[Model 7.5.8, p. 386]{BarJu95}. We will prove the following theorem.
\begin{theorem}
In the Cohen model the following equalities hold. 

\begin{enumerate}
\item
For all $R \in \{\in^*, \leq^*, \neq^*\}$ we have $\mfb (R_\mathcal K) = \mfb (R_\Me) = \aleph_1$ 
\item
$\mfb (R_\Null) = \kappa^+$ 
\end{enumerate}
In words, the numbers associated with $\Me$ and $\Kb$ are $\aleph_1$ and those associated with $\Null$ are $\kappa^+$.
\label{cohenmodel}
\end{theorem}


\begin{proof}
There are two things to show. First that $\mfb (\in^*_\mathcal K) = \mfb (\neq^*_\Me) = \aleph_1$ and second that $\mfb (\in^*_\Null) = \mfb (\neq^*_\Null) = \kappa^+$. The first one is easy. Recall that (iterated) Cohen forcing does not add a slalom eventually capturing all the ground model reals, nor an eventually different function and both ${\rm cov}(\Me)$ and $\mfd$ ($= {\rm cov} (\Kb)$) are $\kappa$ so Lemma \ref{lemmab<cov} applies.

The second argument is slightly more involved. We need to verify that there is a $(\kappa, \aleph_1)$-Rothberger family for the null ideal as described in Lemma \ref{coveringfamilylemma}. Let $G$ be $V$-generic for the product of $\kappa$ many Cohen reals and work in $V[G]$. Let $\{c_\alpha \; | \; \alpha < \kappa\}$ enumerate the Cohen reals. Each such real, say $c_\alpha$, codes a null set $N_\alpha$. It is this family $\{N_\alpha \; | \; \alpha < \kappa\}$ that will witness the lemma. It remains to see that any $\aleph_1$-sized subset is covering. 

By the ccc, for each $x:\omega \to \omega$ there is a countable $X \subseteq \kappa$ so that in fact $x$ is added by the forcing restricted to $X$. Let us denote, for $\dot{x}$ a name for a real $X_{\dot{x}}$ the countable support as described above. Observe that if $\alpha \notin X_{\dot x}$ then in fact $\dot{x}$ is forced to be in $N_\alpha$ by mutual genericity of the Cohen reals. But now, given any $\aleph_1$-sized family $\mathcal Y \subseteq \{N_\alpha \; | \; \alpha < \kappa\}$ and any name for a real $\dot{x}$, it must be the case that there is an $N_\alpha \in \mathcal Y$ so that $\alpha \notin X_{\dot{x}}$ so $x \in N_\alpha$ as needed.
\end{proof}


\subsection{The Random Model}
The random model is the model obtained by forcing with the random forcing $\mathbb B_\kappa$ for the measure algebra $2^\kappa$ over a model of $\GCH$, see \cite[Model 7.6.8, p. 393]{BarJu95}. We will prove the following theorem. 

\begin{theorem}
In the random model the following hold. 
\begin{enumerate}
\item
For all $R \in \{\in^*, \leq^*, \neq^*\}$ we have $\mfb (R_\Me)= \kappa^+$
\item
$\mfb (\in^*_\Null) = \mfb(\leq^*_\Null) = \aleph_1$
\item
$\mfb (\neq^*_\Null) = \mfb(\neq^*_\Kb) =  \kappa$
\item
$\mfb (\in^*_\Kb) = \mfb(\leq^*_\Kb) = \aleph_2$
\end{enumerate}
\label{random}
\end{theorem}

Observe from this theorem the somewhat surprising constellation given by looking at the cardinals associated with $\in^*$ and $\leq^*$. Namely we have that $\mfb(\in^*_\Null) = \aleph_1 < \mfb(\in^*_\Kb) = \aleph_2 < \mfb (\in^*_\Me) = \kappa^+$, and similarly for $\leq^*$. We do not know whether it is consistent that the three cardinals related to $\neq^*$ are simultaneously distinct.

\begin{proof}
Let $G \subseteq \mathbb B_\kappa$ be generic over $V$ and enumerate the $\kappa$-many random reals as $\{r_\alpha \; | \; \alpha < \kappa\}$. We argue for each point individually, starting with the cardinals for the meager ideal. This argument is almost verbatim the same as the argument in the case of the cardinals associated to the null ideal in the Cohen model. The difference is that each random real $r_\alpha$ codes a meager set $M_\alpha$.

The second item follows from Lemma \ref{lemmab<cov} noting that random forcing makes ${\rm cov}(\Null) = \kappa$ and is $\baire$-bounding.

For the third, note first that $\mfb(\neq^*)$ ($= {\rm non}(\Me)$) is $\kappa$ in the random model since every random real adds an eventually different real. It follows that $\kappa = \mfb(\neq^*) \leq \mfb(\neq^*_\Null) = \mfb(\neq^*_\Kb)$ so we get the lower bound. Conversely, note that ${\rm non}(\Null) = {\rm non}(\Kb) = \mfb = \aleph_1$ in the random model so by Lemma \ref{omegabounds} we get $\mfb(\neq^*_\Null) \leq \mfb(\neq^*)^{{\rm non}(\Null)} = \kappa^{\aleph_1} = \kappa$ and the same for $\mfb(\neq^*_\Kb)$. 

For the fourth, the lower bound comes from the fact that $\mfd = \aleph_1$ so by Lemma \ref{lemmacov=1} we get $\aleph_2 \leq \mfb (\in^*_\Kb) \leq \mfb(\leq^*_\Kb)$. The upper bound is proved in exactly the same manner as Lemma \ref{lemma4}.
\end{proof}

\subsection{The Sacks Model}
The Sacks model refers to either the countable support product or iteration of $\kappa$ many copies of Sacks forcing, $\mathbb S$, over a model of $\GCH$, where in the later case we must insist $\kappa = \aleph_2$, see \cite[Model 7.6.2, p. 388]{BarJu95}. It turns out that the computation of the cardinals is unchanged regardless of whether we take the product or the iteration. 

\begin{theorem}
For all $R$ and $\mathcal I$, in the side by side or iterated Sacks model $\mfb (R_\mathcal I) = \aleph_2$.
\end{theorem}

\begin{proof}
This follows in a straightforward way from the Sacks property by combining Lemma \ref{lemmacov=1} and Lemmas \ref{lemma4} and \ref{omegabounds}. In particular, the Sacks property ensures that all covering numbers are small so $\aleph_2$ is a lower bound and the hypotheses of Lemma \ref{lemma4} (for the side by side)/the cardinals in Lemma \ref{omegabounds} (for the iteration) are also met thanks to the Sacks property so $\aleph_2$ is the upper bound as well.
\end{proof}

This is somewhat surprising since this means that some higher dimensional cardinals are bigger in the Sacks model than in the Cohen and Random models. 

\subsection{The Hechler Model}
By the Hechler model we mean the finite support iteration of Hechler forcing, $\D$, over a model $\GCH$, see \cite[Model 7.6.9, p. 394]{BarJu95}. 

\begin{theorem}
In the Hechler model the following hold.
\begin{enumerate}
\item
For $R$ equal to $\leq^*$ or $\neq^*$, we have that $\mfb(R_\mathcal I) = \kappa^+$ for all $\mathcal I \in \{\Null, \Me, \Kb\}$. 
\item
$\mfb (\in^*_\Kb) = \mfb (\in^*_\Me) = \aleph_1$
\item
$\mfb (\in^*_\Null) = \kappa^+$
\end{enumerate}
\label{Hechler}
\end{theorem}

\begin{proof}
Let $G$ be generic over $V$ for the iteration $\D_\kappa$. Denote the $\alpha^{\rm th}$ stage of the iteration by $\D_\alpha$ and let $\{d_\alpha \; | \; \alpha < \kappa\}$ be the $\kappa$ many Hechler reals added by $G$. Since each Hechler real is dominating, and also adds a Cohen real, it is well known that in the Hechler model $\mfb = {\rm cov}(\Me) = \kappa$. As a result, Lemma \ref{lemma5} ensures the first part. For the second part, it suffices to observe that in the Hechler model ${\rm cov} (\Me) = \kappa$, but no slalom is added eventually capturing all ground model reals so Lemma \ref{lemmab<cov} applies.

The last item requires more argument and involves verifying the existence of a $(\kappa, \aleph_1)$-Rothberger family as in Lemma \ref{coveringfamilylemma}. Each Hechler real adds a Cohen real, and the null sets coded by these Cohen reals will be our family. Towards proving this fact, fix a partition $\langle I_k\; | \; k \in \omega\rangle$  of $\omega$ into finite intervals with $|I_k| = k+1$. We start with some facts about finite support iterations of $\sigma$-centered forcing notions $\langle \P_\alpha, \dot{\Q}_\alpha\; | \; \alpha < \kappa \rangle$. The first is well-known.

\begin{fact}
Let $\langle \P_\alpha, \dot{\Q}_\alpha\; | \; \alpha < \kappa \rangle$ be a finite support iteration of $\sigma$-centered forcing notions. Let $\alpha \leq \beta \leq \kappa$ and let $\dot{x}$ be a $\P_\beta$-name for an element of $2^\omega$. Then there are $\P_\alpha$-names $\{\dot{x}_n \in 2^\omega \; | \; n \in \omega\}$ so that for all $\P_\alpha$-names $\dot{y}$ forced to be in $2^\omega$ if $\forces_\alpha \forall n \exists^\infty k ( \dot{x}_n \hook I_k = \dot{y} \hook I_k)$ then $\forces_\beta \exists^\infty k (\dot{x} \hook I_k = \dot{y} \hook I_k)$.
\label{x_nfact}
\end{fact}

\begin{proof}[Proof of \ref{x_nfact}]
By working in $V^{\P_\alpha}$ we may assume without loss of generality that $\alpha = 0$. The proof is by induction on $\beta$. The case where $\beta = 0$ is trivial hence we may suppose that $\beta > 0$. First assume that $\beta = \gamma + 1$ for some $\gamma$ and work in $V^{\P_\gamma}$. Let $\mathbb Q_\gamma = \bigcup_{n < \omega} \Q_n$ be a partition of $\Q_\gamma$ into centered pieces. Fix $n, k \in \omega$ and define $x_n \hook I_k$ to be any value so that no $p \in \Q_n$ forces $\dot{x} \hook I_k \neq \check{x}_n \hook I_k$. To see that such a value exists note that since $I_k$ is finite and $\Q_n$ is centered, if for each $s \in 2^{I_k}$ there was a condition in $\Q_n$ forcing $\dot x \hook I_k \neq s$, then a common extension of these conditions would force a contradictory statement.
If $y \in 2^\omega \cap V^{\P_\gamma}$ is equal to $x_n \hook I_k$ for infinitely many $k$ and all $n$ then $\forces_{\gamma + 1} \exists^\infty k \, \check{y} \hook I_k =\dot{x} \hook I_k$. This is because otherwise there would be a natural numbers $n, l \in \omega$ and a condition $p \in \Q_n$ so that $p \forces \forall k > l \, \check{y} \hook I_k \neq \dot{x} \hook I_k$ but by the way we constructed the $x_n$'s, we can find a $q \leq p$ and an $j> l$ so that $y \hook I_j = x_n \hook I_j$ and $q \forces \dot{x} \hook I_j = \check{x}_n \hook I_j$. Now, let $\dot{x}_n$ name $x_n$ in $V$ and apply the inductive hypothesis to each $\dot{x}_n$ to get elements of $2^\omega$, say $\{x_{n, m} \; | \; n, m \in \omega\}$ so that for each $n <\omega$ and $y$ we have that if $\forall m \exists^\infty k \, (y \hook I_k =x_{n, m} \hook I_k)$ then $\forces_\gamma \exists^\infty k (\check{y} \hook I_k = \dot{x}_n \hook I_k)$. Then using the countable set $\{x_{n, m} \; | \; n, m \in \omega\}$ for the $\dot{x}$ then witnesses the fact in this case.

Now suppose that $\beta$ is a limit ordinal. Since $\dot{x}$ names a real, we can assume without loss of generality that $\beta$ has countable cofinality. Let $\langle \gamma_n \; | \; n\in \omega\rangle$ be a strictly increasing sequence of ordinals with limit $\beta$. For each $n < \omega$ find a decreasing set of conditions in $V^{\P_{\gamma_n}}$ deciding all of $\dot{x}$ and let $x_n \in V^{\P_{\gamma_n}}$ be the real interpreting $\dot{x}$ based on these decisions. Let $\dot{x}_n$ be a name for $x_n$ in $V$ and, applying the inductive hypothesis to each $x_n$ let $\{x_{n, m} \; | \; m \in \omega\}$ be a set of reals so that for each $n <\omega$ and $y$ we have that if $\forall m \exists^\infty k \, (y \hook I_k =x_{n, m} \hook I_k)$ then $\forces_{\gamma_n} \exists^\infty k \, (\check{y} \hook I_k = \dot{x}_n \hook I_k)$. We claim that these $\{x_{n, m} \; | \;n,  m \in \omega\}$  work for $\dot{x}$. To see this, suppose that for all $n$ and $m$ there are infinitely many $k$ so that $x_{n, m} \hook I_k = y \hook I_k$ but there is an $l \in \omega$ and a $p \in \P_\beta$ so that $p \forces_\beta \forall k > l \;( \dot{x} \hook I_k \neq \check{y} \hook I_k)$. By the finiteness of the support, $p$ is actually a $\P_{\gamma_n}$ condition for some $n < \omega$ and, by assumption $p \forces_{\gamma_n} \exists^\infty k \, (\dot{x}_n \hook I_k = \check{y} \hook I_k)$. However, by the construction of the $x_n$'s we can find an $r \leq_{\beta} p$ and a $k > l$ so that $r \forces_{\beta} \dot{x} \hook I_k = \dot{x}_n \hook I_k = \check{y} \hook I_k$ which is a contradiction.
\end{proof}

Now for some $\beta\leq \kappa$ let $\dot{x}$ be a $\D_\beta$-name for an element of $2^\omega$ and define inductively on $\beta \leq \kappa$ the {\em hereditary support} of $\dot{x}$, denoted ${\rm supp}(\dot{x})$, as follows. If $\beta = 0$ then ${\rm supp}(\dot{x}) = \emptyset$. If $\beta = \gamma + 1$, let $(\dot{x}_n\; | \; n \in \omega)$ be the $\D_\gamma$-names constructed from $\dot{x}$ as in Fact \ref{x_nfact} and let $\{p_{m, k}\; | \; k \in \omega\}$ be a maximal antichain deciding $\dot{x}(\check{m})$. Set 

\begin{equation*}
{\rm supp}(\dot{x}) = \bigcup_{n < \omega} {\rm supp}(\dot{x}_n) \cup \bigcup_{m, k}{\rm supp} (p_{m, k})
\end{equation*}
If $\beta$ is a limit of countable cofinality, let $(\gamma_n \; | \; n< \omega)$ be a strictly increasing sequence of ordinals whose limit is $\beta$, for each $n < \omega$ let $\{\dot{x}_{n, m} \; | \; m \in \omega \}$ be the $\D_{\gamma_n}$-names constructed from $\dot{x}$ as in Fact \ref{x_nfact} and let

\begin{equation*}
{\rm supp}(\dot{x}) = \bigcup_{n, m < \omega} {\rm supp}(\dot{x}_{n, m}) \cup \bigcup_{m, k}{\rm supp}(p_{m, k})
\end{equation*}
where $p_{m, k}$ are as before. Finally for $\beta$ a limit ordinal of uncountable cofinality observe that, by the finite support, there is a $\gamma < \beta$ so that $\dot{x}$ is in fact equivalent to a $\D_\gamma$-name. Let  ${\rm supp}(\dot{x})$ be the support of this $\D_\gamma$-name. Note that in all cases ${\rm supp}(\dot{x})$ is a countable set of ordinals. We need another fact.
\begin{fact}
Let $\alpha < \alpha ' \leq \beta \leq \kappa$ and let $\dot{x}$ be a  $\D_\beta$-name for an element of $2^\omega$. Assume ${\rm supp}(\dot{x}) \cap [\alpha, \alpha ') = \emptyset$. Then there are $\D_\alpha$-names for elements of $2^\omega$, $\{\dot{x}_n\; | \; n \in \omega \}$, so that for all $\D_{\alpha '}$-names $\dot{y}$ for elements of $2^\omega$, if $\forces_{\alpha '} \forall n \exists^\infty k \, (\dot{x}_n \hook I_k = \dot{y} \hook I_k)$, then $\forces_\beta \exists^\infty k \, (\dot{x} \hook I_k = \dot{y} \hook I_k)$.
\label{x_nfact2}
\end{fact}

\begin{proof}[Proof of \ref{x_nfact2}]
The proof is by induction on $\beta$.  If $\beta = \alpha '$ or $\beta = 0$ then the fact is trivially true so assume $0 \leq\alpha < \alpha ' < \beta$. There are two cases corresponding to whether $\beta$ is a successor or a limit. 

First suppose that $\beta = \gamma + 1$ for some $\gamma$. Let $\{\dot{x}_n\; | \; n \in \omega \}$ be $\D_\gamma$-names as constructed in the proof of Fact \ref{x_nfact}. By the definition of the support plus the assumption that ${\rm supp}(\dot{x}) \cap [\alpha, \alpha ') = \emptyset$ we get that ${\rm supp}(\dot{x}_n) \cap [\alpha, \alpha ') = \emptyset$ for all $n < \omega$. Applying the inductive hypothesis we get for each $n < \omega$ a countable set of $\D_\alpha$-names $\{\dot{x}_{n, m} \; | \; m \in \omega\}$ so that for all $\D_{\alpha '}$-names $\dot{y}$ for elements of $2^\omega$, if $\forces_{\alpha '} \forall m \exists^\infty k \, (\dot{x}_{n, m} \hook I_k = \dot{y} \hook I_k)$, then $\forces_\beta \exists^\infty k \,( \dot{x}_n \hook I_k = \dot{y} \hook I_k)$. Using these the same way as in the proof of Fact \ref{x_nfact} completes the proof of this case.

Suppose now that $\beta$ is a limit ordinal. Since $\dot{x}$ is a $\D_\beta$-name for a real, we can reduce to the case where $\beta$ has countable cofinality. Let $\langle \gamma_n \; | \; n < \omega\rangle$ be a strictly increasing sequence of ordinals whose limit is $\beta$. For each $n < \omega$ let $\{\dot{x}_{n, m} \; | \; m \in \omega \}$ be the $\D_{\gamma_n}$-names constructed from $\dot{x}$ as in Fact \ref{x_nfact}. Using these countably many names the rest of the proof of this case is the same as in the successor case.
\end{proof}

Note that if $y \in 2^\omega$ then the set $N_y = \{x \in 2^\omega \; | \; \exists^\infty k \, (x\hook I_k = y \hook I_k) \}$ is a null set. Let $c_\alpha$ be the Cohen real added by the $\alpha^{\rm th}$-Hechler real. 
\begin{claim}
If $\dot{x}$ is a $\P_\kappa$-name for an element of $2^\omega$ and $\alpha \notin {\rm supp}(\dot{x})$ then $\forces_\kappa \dot{x} \in N_{\dot{c}_\alpha}$. 
\end{claim}

\begin{proof}[Proof of Claim]
Fix $\alpha < \kappa$. It's well known that for all $\D_\alpha$-names $\dot{z}$ for elements of $2^\omega$ we have $\forces_{\alpha + 1} \exists^\infty k \, (\dot{z} \hook I_k = \dot{c}_\alpha \hook I_k)$. Now the claim follows by applying Fact \ref{x_nfact2} to $\alpha ' = \alpha + 1$ and $\beta = \kappa$.
\end{proof}

The proof is now essentially the same as that for Cohen forcing. Since ${\rm supp}(\dot{x})$ is countable any $\aleph_1$-sized subfamily of $\{N_{c_\alpha} \; | \; \alpha < \kappa\}$ must be covering.
\end{proof}


Let us note for later that the computation that $\mfb(\in^*_\Null) = \kappa^+$ relied solely on the facts that Hechler forcing is $\sigma$-centered and adds Cohen reals.

\subsection{The Dual Random Model}
For any uncountable $\lambda \leq \kappa$, the $\lambda$-dual random model is formed by first forcing ${\rm add}(\Null) = \mfc = \kappa$ by a ccc forcing and then adding $\lambda$-many random reals. The result, regardless of which $\lambda$ is chosen, the $\lambda$-dual random model is a model where the values of the Cicho\'{n} diagram are determined by ${\rm non}(\Null) = \aleph_1$ and ${\rm cov}(\Null) = \mfb = \kappa$.
\begin{theorem}
In the $\lambda$-dual random model, the following equalities hold.
\begin{enumerate}
\item
$\mfb(\in^*_\Null) = \mfb(\in^*_\Kb) = \aleph_1$
\item
$\mfb(\leq^*_\Null) = \mfb(\neq^*_\Null) = \kappa$
\item
$\mfb(\leq^*_\Kb) = \mfb(\neq^*_\Kb) = \mfb(\leq^*_\Me) = \mfb(\neq^*_\Me) = \kappa^+$
\item
$\mfb(\in^*_\Me) = \lambda^+$
\end{enumerate}
\end{theorem}

Let us note one interesting feature of this model. The value $\mfb(\in^*_\Me)$ depends on which $\lambda$ we chose, even though the cardinals in the Cicho\'{n} diagram are the same regardless of $\lambda$. It follows that the $\mfb(R_\mathcal I)$ numbers are not uniquely determined by the values of the cardinal characteristics on $\omega$ alongside cardinal arithmetic. 

\begin{proof}
For item 1 we can apply Lemma \ref{lemmab<cov} since $\mfb(\in^*) = \aleph_1$ but the covering numbers for both $\Kb$ ($=\mfd$) and $\Null$ are of size continuum.

For item 2 note we apply both inequalities in Lemma \ref{omegabounds}. On the one hand $\kappa$ is a lower bound since $\mfb= \mfb(\neq^*) = \kappa$. On the other hand, $\kappa$ is also an upper bound since $\mfb^{{\rm non}(\Null)} = \kappa^{\aleph_1} = \kappa$.

Item 3 follows from Lemma \ref{lemma5} since $\mfb = {\rm non}(\Me) = \kappa = 2^{\aleph_0}$. 

Finally we tackle $\mfb(\in^*_\Me)$. By the same argument used in Theorem \ref{random}, $\mfb(\in^*_\Me) > \lambda$ as the $\lambda$ random reals code $\lambda$ many meager sets forming a $(\lambda, \aleph_1)$-Rothberger family. Also, for $\lambda = \kappa$, $\mfb  (\in^*_\M) \leq 2^\mfc = \mfc^+ = \kappa^+ = \lambda^+$.

So assume $\lambda < \kappa$. We will show that $\bb (\in^*_\M) \leq \lambda^+$. We mostly work in the final extension but once
step back into an intermediate random extension. Let $\{ x_\alpha \; | \; \alpha < \kappa \}$ list all reals. Let $\{ (A_\alpha , M_\alpha ) \; | \;
\alpha < \kappa \}$ list all pairs $(A,M)$ such that  $A \in [\lambda^+]^{\aleph_1}$ and $M$ is a Borel meager set. Note that by our assumption
$(\lambda^+)^{\aleph_1} \leq \kappa^{\aleph_1} = \kappa$, so this is possible. Also let $\{ y_\xi \; | \;  \xi < \omega_1 \} \subseteq \omom$
be a sequence witnessing $\bb (\in^*) = \aleph_1$.

Now construct functions $\{ f_\gamma \; | \; \gamma < \lambda^+ \} \subseteq (\omom)^{\omom}$ and a strictly increasing sequence
$\{ \zeta_\alpha \; | \;  \alpha < \kappa \}$ of ordinals in $\kappa$ as follows. Suppose we are at step $\alpha$. Find $\zeta_\alpha$ such that
\begin{itemize}
\item $\zeta_\alpha > \zeta_\beta$ for $\beta < \alpha$,
\item $x_{\zeta_\alpha} \notin M_\alpha$.
\end{itemize}
Then define $f_\gamma (x_{\zeta_\alpha}) $ for $\gamma \in A_\alpha$ such that
\[ \{ f_\gamma (x_{\zeta_\alpha}) \; | \;  \gamma \in A_\alpha \} = \{ y_\xi \; | \;  \xi < \omega_1 \}. \]
Define the remaining values $f_\gamma (x_\beta)$ arbitrarily. This completes the construction of the $f_\gamma$.

We claim $\{ f_\gamma \; | \; \gamma < \lambda^+ \}$ is a witness for $\bb (\in^*_\M)$.

For assume this were not the case. Then we could find $\varphi: \omom \to \mathcal S$  such that
\[ \{ x_\alpha \; | \; \alpha < \kappa \, \&\, f_\gamma (x_\alpha) \notin^* \varphi (x_\alpha) \} \in \M \]
for all $\gamma < \lambda^+$. For each such $\gamma$ let $B_\gamma$ be a Borel meager set such that 
\[ \forall x \notin B_\gamma \; (f_\gamma (x) \in \varphi (x) ). \]
There is a countable set $X_\gamma \subseteq \lambda$ such that the code of $B_\gamma$ lies in the $X_\gamma$-extension
(that is, the extension obtained by only adding the random reals with index in $X_\gamma$).
Note that $\cof ( [ \lambda]^{\leq \aleph_0} ) = \lambda$.\footnote{This is true in the ground model and not
changed by ccc forcing. In fact, it is true in ZFC for the $\aleph_n$, and forcing its failure for larger $\lambda$
needs large cardinals.}   Hence we can find a countable $X \subseteq \lambda$ and a $Y \subseteq \lambda^+$ of size $\lambda^+$ s.t.
$X_\gamma \subseteq X$ for all $\gamma \in Y$; in particular, $B_\gamma$ is coded in the $X$-extension for all $\gamma \in Y$.
We may assume $Y$ also belongs to the $X$-extension.
Let $A \subseteq Y$ be any subset of size $\aleph_1$ in the $X$-extension. Since ${\rm add} (\Null) = {\rm add} (\Me) = \kappa = \cc \geq \aleph_2$ in the
$X$-extension, the union $\bigcup_{\gamma \in A} B_\gamma$ must be meager, so there is a Borel meager set $M$, still
coded in the $X$-extension, such that $\bigcup_{\gamma \in A} B_\gamma \subseteq M$. By construction, there is $\alpha < \kappa$ 
such that  $(A,M) = (A_\alpha, M_\alpha)$. Since $x_{\zeta_\alpha} \notin M_\alpha$ we must have
\[ f_\gamma (x_{\zeta_\alpha}) \in \varphi (x_{\zeta_\alpha}) \]
for all $\gamma \in A_\alpha$. This clearly contradicts the fact that 
\[ \{ f_\gamma (x_{\zeta_\alpha}) \; | \;  \gamma \in A_\alpha \} = \{ y_\xi \; | \;  \xi < \omega_1 \} \]
is a witness for $\bb (\in^*)$, and the proof of the theorem is complete.
\end{proof}

\subsection{The Random/Hechler Model}
The random/Hechler model refers to the finite support iteration of length $\kappa$ which results from alternating between adding random reals and Hechler reals. 

\begin{theorem}
In the random/Hechler model the following equalities hold.
\begin{enumerate}
\item
For all $\mathcal I \in \{\Null, \Me, \Kb\}$ we have $\mfb(\in^*_\mathcal I) = \aleph_1$
\item
All other cardinals of the form $\mfb(R_\mathcal I)$ are equal to $\kappa^+$.
\end{enumerate}
\end{theorem}

\begin{proof}
The values of the Cicho\'{n} diagram in the random/Hechler model are that all cardinals are $\kappa$ with the exception of ${\rm add} (\Null) = \mfb(\in^*) = \aleph_1$. It follows from Lemma \ref{lemmab<cov} that $\mfb(\in^*) = \mfb(\in^*_\mathcal I) = \aleph_1$ for all $\mathcal I \in \{\Null, \Me, \Kb\}$ and from Lemma \ref{lemma5} that all the other cardinals are $\kappa^+$.
\end{proof}

\subsection{The Eventually Different Model}
Recall that eventually different forcing $\mathbb{E}$ consists of pairs $(s, E)$ so that $s \in \omega^{<\omega}$ and $E \subseteq \baire$ is finite. The extension relation is $(s_0, E_0) \leq (s_1, E_1)$ just in case $s_0 \supseteq s_1$, $E_0 \supseteq E_1$ and for all $k \in {\rm dom}(s_0) \setminus {\rm dom}(s_1)$ we have $s_0(k) \neq f(k)$ for all $f \in E_1$. The eventually different model is the model obtained by adding $\kappa$-many eventually different reals with finite support over a model $\GCH$, see \cite[Model 7.5.6, p. 385]{BarJu95}. 

\begin{theorem}
In the eventually different model the following equalities hold. 
\begin{enumerate}
\item
For $R$ equal to $\in^*$ or $\leq^*$ we have that $\mfb(R_\Me) = \mfb(R_\Kb) = \aleph_1$
\item
$\mfb(\neq^*_\Me) = \kappa^+$
\item
$\mfb(\neq^*_\Kb) = \kappa$
\item
For all $R \in \{\in^*, \leq^*, \neq^*\}$ we have $\mfb(R_\Null) = \kappa^+$
\end{enumerate}
\end{theorem}

\begin{proof}
Recall that the values of the Cicho\'{n} diagram in the eventually different model are determined by $\mfb = {\rm cov}(\Null) = \aleph_1$ and ${\rm non}(\Me) = \mfb(\neq^*) = {\rm cov}(\Me) = \kappa$. From this item 1 follows from Lemma \ref{lemmab<cov} and item 2 follows from Lemma \ref{lemma5}. The third item follows from Lemma \ref{omegabounds} since $\mfb(\neq^*) = \kappa$ and ${\rm non} (\Kb) = \mfb = \aleph_1$. The fourth item is verbatim the same as the proof that $\mfb(\in^*_\Null) = \kappa^+$ holds in the Hechler model, see Theorem \ref{Hechler}. This follows from the observation made after that proof plus the fact that $\mathbb E$ is $\sigma$-centered and adds Cohen reals.
\end{proof}

\subsection{The Laver Model}
The Laver model is the $\aleph_2$ length countable support iteration of Laver forcing over a model of $\GCH$, see \cite[Model 7.6.13, p. 396]{BarJu95}. 
\begin{theorem}
In the Laver model the following hold.
\begin{enumerate}
\item
For $R$ equal to $\leq^*$ or $\neq^*$ we have $\mfb(R_\Kb) = \mfb(R_\Me) = \aleph_3$.
\item
For all $R \in \{\in^*, \leq^*, \neq^*\}$ we have $\mfb(R_\Null) = \aleph_2$
\item
$\mfb(\in^*_\Kb) = \aleph_1$

\end{enumerate}
\end{theorem}

Note that the one cardinal not determined by this theorem is $\mfb(\in^*_\Me)$. This is the one cardinal that we were not able to calculate in one of the ``standard iterated models". By the fact that ${\rm cov}(\Me) = \aleph_1$ in the Laver model, Lemma \ref{lemmacov=1} implies the value is either $\aleph_2$ or $\aleph_3$ but we do not know which one. To show that it is $\aleph_3$ it would suffice to exhibit an $(\aleph_3, \aleph_1)$-Rothberger family for $\Me$ however we do not know if one exists in the Laver model.

\begin{proof}
Recall that the values of the Cicho\'{n} diagram in the Laver model are determined by ${\rm cov}(\Null) = {\rm non}(\Null) = \aleph_1$ and $\mfb = \aleph_2$. Let $\vec{l} = \{l_\alpha \; | \; \alpha < \aleph_2\}$ enumerate the Laver reals added and work in $V[\vec{l}]$. The first item then follows from Lemma \ref{lemma5}. For the second item, the lower bound comes from Lemma \ref{lemmacov=1}. For the upper bound, observe that by Lemma \ref{omegabounds} we have $\mfb(\neq^*_\Null) \leq \mfb(\neq^*)^{{\rm non}(\Null)} = \aleph_2^{\aleph_1} = \aleph_2$. The third item follows from Lemma \ref{lemmab<cov}, recalling that ${\rm cov}(\Kb) = \mfd$ which is $\aleph_2$ in the Laver model. 
\end{proof}

\subsection{The Miller Model}
The Miller model is the $\aleph_2$-length countable support iteration of Miller forcing over a model of $\GCH$, see \cite[Model 7.5.2, p. 382]{BarJu95}. 
\begin{theorem}
In the Miller model the following hold.
\begin{enumerate}
\item
For all $R \in \{\in^*, \leq^*, \neq^*\}$ we have $\mfb(R_\Me) = \mfb(R_\Null) =\aleph_2$
\item
For all $R \in \{\in^*, \leq^*, \neq^*\}$ we have $\mfb(R_\Kb) = \aleph_1$
\end{enumerate}
\end{theorem}

\begin{proof}
Recall that the values of the Cicho\'{n} diagram in the Miller model are determined by $\mfd = \aleph_2$ and ${\rm non}(\Null) = {\rm non}(\Me) = \aleph_1$. The first item then follows by Lemmas \ref{lemmacov=1} (for the lower bound) and \ref{omegabounds} (for the upper bound). The second item follows from Lemma \ref{lemmab<cov}. 
\end{proof}

\section{$\ZFC$ Results for the Dominating Numbers}
In this section we explore $\ZFC$ provable equalities between the cardinals of the form $\mfd(R_\mathcal I)$. The main theorem of this section is the following, which shows that, unlike the bounding numbers, $\mfd(R_\mathcal I)$ does not depend on $\mathcal I$. Note that this strengthens Main Theorem \ref{mainthm4} and provides a proof of that theorem. We also investigate the relationship between the $\mfd(R_{\{\emptyset\}})$ cardinals and the cardinals $\mfd^\lambda_\kappa$ introduced in \cite{Br19}.

\begin{theorem}
Assume $\mathcal I$ is an ideal on $\baire$ so that $cof(\mathcal I) \leq \mfc$ and for all $X \in \mathcal I$ $|\baire \setminus X| = \mfc$. Then $\mfd(R_\mathcal I) = \mfd(R_{\{\emptyset\}})$ for all $R \in \{\in^*, \leq^*, \neq^*\}$. Here $\{\emptyset\}$ is the trivial ideal consisting only of the empty set.
\label{idealdoesntmatter}
\end{theorem}

Note that $\Null$, $\Me$ and $\Kb$ all fulfill the hypotheses on $\mathcal I$ in the theorem statement. This theorem is somewhat analogous to the fact that e.g. the dominating number (on $\omega$) is unchanged if we insist on everywhere dominating versus mod finite domination. Note that there as well this fact does not hold true of unbounded sets.

\begin{proof}
Fix an ideal $\mathcal I$ as in the statement of the theorem. Clearly any $R_{\{\emptyset\}}$-dominating family is $R_\mathcal I$-dominating hence $\mfd(R_\mathcal I) \leq \mfd (R_{\{\emptyset\}})$ so we need to just prove the reverse inequality. Fix $\kappa < \mfd(R_{\{\emptyset\}})$ and a family of $\kappa$ many functions $\mathcal F = \{f_\alpha \; | \; \alpha < \kappa\}$. We need to see that $\mathcal F$ is not dominating. 

Enumerate (possibly with repetitions) a cofinal family in $\mathcal I$, $\{X_\alpha \; | \; \alpha < 2^{\aleph_0}\}$. Also, let $\bar{\gamma}:2^{\aleph_0} \times 2^{\aleph_0} \to 2^{\aleph_0}$ be a bijection. Inductively on $\gamma < 2^{\aleph_0}$ for $\gamma = \bar{\gamma}(\alpha, \beta)$ define $y_{\alpha, \beta} \in \baire$ so that $y_{\alpha, \beta} \neq y_{\alpha ' , \beta '}$ for $(\alpha, \beta) \neq (\alpha ', \beta ' )$ and $y_{\alpha, \beta} \notin X_\alpha$ for all $\beta < 2^{\aleph_0}$. Note that this is possible by the assumption that $\baire \setminus X_\alpha$ has size continuum for all $\alpha$.

For each $\alpha < 2^{\aleph_0}$ consider the restriction of $\mathcal F$ to $\{y_{\alpha, \beta} \; | \; \beta < 2^{\aleph_0}\}$ i.e. $\{f_\gamma \hook \{y_{\alpha, \beta} \; | \; \beta < 2^{\aleph_0}\} \; | \; f_\gamma \in \mathcal F\}$. Since $\kappa < \mfd (R_{\{\emptyset\}})$, none of these sets are dominating so there is a $g:\baire \to \baire$ so that for all $f \in \mathcal F$ and all $\alpha < 2^{\aleph_0}$ there is a $\beta$ so that $\neg(g(y_{\alpha, \beta}) R f (y_{\alpha, \beta}))$. But since the $X_\alpha$'s formed a basis for $\mathcal I$ it follows that $g$ is not $R_\mathcal I$-bounded by any $f \in \mathcal F$, as needed.
\end{proof}

As a result of this theorem we have the following.
\begin{corollary}
For all $R$ and $\mathcal I$ we have $\mfd(R_\mathcal I) > 2^{\aleph_0}$. 
\label{lemmamfd}
\end{corollary}

\begin{proof}
In light of Theorem \ref{idealdoesntmatter} it suffices to show that for all $R$, $\mfd(R_{\{\emptyset\}} )> 2^{\aleph_0}$. Fix a family $\mathcal F = \{f_\alpha \; | \; \alpha < 2^{\aleph_0}\} \subseteq \bbb$ (or $\mathcal S^{\baire}$ in the case of $R = \in^*$). Enumerate $\baire$ as $\{x_\alpha \; | \; \alpha < 2^{\aleph_0}\}$. Define $g:\baire \to \baire$ so that $g(x_\alpha)$ is not $R$ below $f_\alpha(x_\alpha)$. It follows that for each $\alpha$ there is an $x$ so that $\neg(g(x) R f_\alpha(x))$ and hence $\mathcal F$ does not contain a bound on $g$ so it is not dominating. 
\end{proof}

In the next section we study consistent inequalities between the dominating numbers of the relations $R_\mathcal I$. The computations of these cardinals in various models factor through relating the cardinals $\mfd(R_{\{\emptyset\}})$ to the cardinals $\mfd_\kappa^\lambda$ of \cite[Section 4]{Br19}. In the rest of this section we establish $\ZFC$ results between these two families of cardinals. Recall that if $\lambda \geq \kappa$ and $f, g \in \kappa^\lambda$ then we let $f \leq^* g$ if and only if $\{\alpha < \lambda \; | \; f(\alpha) > g(\alpha)\}$ has size less than $\kappa$. The cardinal $\mfd^\lambda_\kappa$ is the dominating number of this relation. The following, which was proved in \cite{Br19}, will be useful for us.

\begin{fact}[Proposition 12 of \cite{Br19}]
The value of $\mfd^\lambda_\kappa$ is unchanged if we work with the total domination relation as opposed to the mod ${<} \kappa$ domination.
\end{fact}

The main question, which is open, is whether every cardinal $\mfd(R_{\{\emptyset\}})$ is provably equal to a cardinal of the form $\mfd^\lambda_\kappa$. Indeed, in every model we know the value of $\mfd(R_{\{\emptyset\}})$ is equal to that of $\mfd^\mfc_{\mfb(R)}$. We know that this equality follows from certain extra assumptions about $R$. To explain this result we need a few more definitions.
\begin{definition}
Fix a relation $R \in \{\in^*, \leq^*, \neq^*\}$ and an arbitrary cardinal $\kappa$.

\begin{enumerate}
\item
A family $\{x_\alpha \; | \; \alpha < \kappa\} \subseteq \baire$ is an {\em eventually} $R$-{\em dominating} sequence if for all $y \in \baire$ there is an $\alpha_0 < \kappa$ so that for all $\alpha \in [\alpha_0, \kappa)$, $yRx_\alpha$.
\item
A family $\{x_\alpha \; | \; \alpha < \kappa\} \subseteq \baire$ is an {\em eventually} $R$-{\em unbounded} sequence if for all $y \in \baire$ there is an $\alpha_0 < \kappa$ so that for all $\alpha \in [\alpha_0, \kappa)$, $\neg(x_\alpha R y)$.
\end{enumerate}

\end{definition}

It is not hard to check that these two notions are dual to one another i.e. $R$-eventually dominating is $\neg \check{R}$-eventually unbounded where $\neg \check{R}$ is the dual relation to $R$. Note also that the existence of an eventually $R$-dominating sequence of length $\kappa$ implies that $\mfb(R) \geq cf(\kappa)$ and $\mfd(R) \leq cf(\kappa)$. Dually, the existence of an eventually $R$-unbounded sequence of length $\kappa$ implies $\mfb(R) \leq cf(\kappa)$ and $\mfd(R) \geq cf(\kappa)$. Some examples will be helpful also moving forward.

\begin{example}
\begin{enumerate}
\item
In $\ZFC$ there is an eventually $\leq^*$-unbounded sequence of length $\mfb$. Meanwhile the existence of an eventually $\leq^*$-dominating sequence of length $\kappa$ implies that $\mfb = \mfd = cf(\kappa)$. Conversely if $\mfb = \mfd = \kappa$ then there is an eventually $\leq^*$-dominating sequence of length $\kappa$.

\item
In the Cohen model there are eventually $\neq^*$-unbounded sequences of length $\kappa$ for $\aleph_1 \leq \kappa \leq \mfc$ but no eventually $\neq^*$-dominating sequences of any length.

\item
In the random model there are eventually $\neq^*$-dominating sequences of length $\kappa$ for $\aleph_1 \leq \kappa \leq \mfc$ but no eventually $\neq^*$-unbounded sequences of any length.
\end{enumerate}
 
\end{example}

\begin{lemma}
Fix $R \in \{\in^*, \leq^*, \neq^*\}$ and a cardinal $\kappa \leq \mfc$.
\begin{enumerate}
\item
If there is an eventually $R$-dominating sequence of length $\kappa$ then $\mfd(R_{\{\emptyset\}}) \leq \mfd^\mfc_{\kappa}$.

 \item
If there is an eventually $R$-unbounded sequence of length $\kappa$ then $\mfd(R_{\{\emptyset\}}) \geq \mfd^\mfc_{\kappa}$.
\end{enumerate}
\label{eventualsequences}
\end{lemma}

\begin{proof}
Let $\mathcal F = \{x_\alpha \; | \; \alpha < \kappa\} \subseteq \baire$ be an eventually $R$-dominating sequence of length $\kappa$. Let $\mathcal G$ be a witness for $\mfd^\mfc_\kappa$. Also let $\{y_\gamma \; | \; \gamma < \mfc\}$ enumerate $\baire$. For $g \in \mathcal G$ define $g' \in \bbb$ so that $g'(y_\gamma) = x_{g(\gamma)}$. It follows that $\mathcal G ' =\{g' \; | \; g \in \mathcal G\}$ is a $R_{\{\emptyset\}}$-dominating family, as needed.

The second item is proved the same way by duality.
\end{proof}

As a corollary of this lemma we have the following in the case that $R = \leq^*$. Note that this relation is particularly nice because it is a partial order (when moded out by equality ``mod finite") and hence there is in $\ZFC$ an eventually $\leq^*$-unbounded family of size $\mfb$ and therefore $\mfd(\leq^*_{\{\emptyset\}}) \geq \mfd^\mfc_\mfb$.

\begin{lemma}
If $\mfb = \mfd$ then $\mfd(\leq^*_{\{\emptyset\}}) = \mfd^\mfc_\mfb$.
\label{scale}
\end{lemma}

\begin{proof}
If $\mfb = \mfd$ there is a scale of length $\mfb$, i.e. a dominating family well-ordered by $\leq^*$ of length $\mfb$, see \cite[Theorem 2.6]{BlassHB}. Such a scale is simultaneously an eventually $\leq^*$-unbounded and an eventually $\leq^*$-dominating sequence. Hence, by combining the first and second items in Lemma \ref{eventualsequences} it follows that $\mfd(R_{\emptyset\}})$ is equal to $\mfd^\mfc_\mfb$. 
\end{proof}

It is unclear whether similar hypotheses to $\mfb = \mfd$ imply the same for the other relations. It is also unclear if the additional assumption is necessary. In every model we have computed $\mfd(R_{\{\emptyset\}})$, regardless of whether or not $\mfb = \mfd$, we have equality. We conjecture this is a $\ZFC$ phenomenon.
\begin{conjecture}
In $\ZFC$ it is provable that $\mfd^\mfc_\mfb = \mfd(\leq^*_{\{\emptyset\}})$.
\end{conjecture}

This conjecture is true if $\mfc < \aleph_\omega$.

\begin{theorem}
If $\mfc = \mfb^{+n}$ for some $n < \omega$ then $\mfd^\mfc_\mfb = \mfd(\leq^*_{\{\emptyset\}})$.
\label{mfc=mfb+n}
\end{theorem}

\begin{proof}
Let $M$ be a model of a large enough fragment of $\ZFC$ of size $\mfd^\mfc_\mfb$ containing a witness for $\mfd^\mfc_\mfb$. We show that $M$ contains a witness for $\mfd(R_{\{\emptyset\}})$ as well. Since the $\mfd^\mfc_\kappa$ are decreasing in $\kappa$ (at least at successor steps), see \cite[Theorem 13]{Br19}, we may assume that $M$ also contains a witness for each $\mfd^\mfc_\kappa$ for $\mfb \leq \kappa \leq \mfc$. Note that $\baire \subseteq M$. Let $\{y_\gamma \; | \; \gamma < \mfc\} \in M$ enumerate $\baire$. By induction on $\kappa$ we show the following:

\begin{center}
$(*)$: For any $\kappa \in [\mfb, \mfc]$ if $\{x_{\gamma, \alpha} \; | \; \gamma < \mfc, \alpha < \kappa\} \in M$ is a list of reals and $f:\mfc \to \kappa$ is arbitrary, then there is a $g:\baire \to \baire$ in $M$ so that for a $\gamma < \mfc$, $x_{\gamma, f(\gamma)} \leq^* g(y_\gamma)$. 
\end{center}

This is clearly true for $\kappa = \mfb$. Furthermore if $(*)$ holds for $\kappa = \mfc$ we're done. Therefore assume $(*)$ holds for $\kappa$ and we will show it holds for $\kappa^+$. Let $\{x_{\gamma, \alpha} \; | \; \gamma < \mfc, \alpha < \kappa^+\} \in M$ be a list of reals. Let $f:\mfc \to \kappa^+$ be arbitrary and let $h:\mfc \to \kappa^+$ in $M$ dominate $f$ everywhere. Reindex the list in $M$ so that for each $\gamma$ we have $\{x'_{\gamma, \alpha} \; | \; \alpha < \kappa\} = \{x_{\gamma, \alpha}\; | \; \alpha < h(\gamma)\}$. Applying the induction hypothesis to $\{x'_{\gamma, \alpha} \; | \; \gamma < \mfc, \alpha < \kappa\} \in M$ and $f':\mfc \to \kappa$ given by $x'_{\gamma, f'(\gamma)} = x_{\gamma, f(\gamma)}$ gives an $g \in M$ so that for all $\gamma < \mfc$, $x'_{\gamma, f'(\gamma)} = x_{\gamma, f(\gamma)} \leq^* g(y_\gamma)$ as needed.
\end{proof}

\section{Consistency Results for the Dominating Numbers}
In this section we consider consistency results for the cardinals of the form $\mfd(R_\mathcal I)$. As shown in Theorem \ref{idealdoesntmatter} the ideal here is unimportant so we really only have three cardinals, $\mfd(\in^*_{\{\emptyset\}})$, $\mfd(\leq^*_{\{\emptyset\}})$, $\mfd(\neq^*_{\{\emptyset\}})$. Nevertheless we can separate them all.

\begin{theorem}
The following constellations of cardinals are all consistent.
\begin{enumerate}
\item
$\mfd(\neq^*_{\{\emptyset\}}) < \mfd(\leq^*_{\{\emptyset\}}) = \mfd(\in^*_{\{\emptyset\}})$ 
\item
$\mfd(\neq^*_{\{\emptyset\}}) = \mfd(\leq^*_{\{\emptyset\}}) < \mfd(\in^*_{\{\emptyset\}})$  
\item
$\mfd(\neq^*_{\{\emptyset\}})< \mfd(\leq^*_{\{\emptyset\}}) < \mfd(\in^*_{\{\emptyset\}})$ 
\item
$\mfd(\neq^*_{\{\emptyset\}}) = \mfd(\leq^*_{\{\emptyset\}}) = \mfd(\in^*_{\{\emptyset\}}) < 2^{\mfc}$ 
\item
$\mfc^+ < \mfd(\neq^*_{\{\emptyset\}}) = \mfd(\leq^*_{\{\emptyset\}}) = \mfd(\in^*_{\{\emptyset\}})$ 
\end{enumerate}
\label{dcardinals}
\end{theorem}

The models witnessing Theorem \ref{dcardinals} are discussed in the proofs of Lemmas \ref{randomdcardinal}, \ref{hechlerdcardinal} and \ref{mixdcardinal} below. The general form of proving these consistency results involves first adding Cohen subsets to $\omega_1$ and then adding various kinds of reals. The point will be that all of the $\mfd(R_{\{\emptyset\}})$ cardinals will reduce to combinatorics on some space of the form $\kappa^\mfc$ but depending on how many subsets of $\omega$ and $\omega_1$ we add, as well as which type of reals we add, $\kappa$ may change depending on the relation $R$. Towards proving these results, we begin with a simple, well known lemma.

\begin{lemma}
Assume $\GCH$ and let $\kappa > \aleph_3$ be a regular cardinal. Let $G \subseteq add(\omega_1, \kappa)$ be generic over $V$. Then in $V[G]$ we have that $\mfd_{\omega_1} = \kappa$ and $\mfd_{\omega_2} = \aleph_3$.
\label{cohensubsets}
\end{lemma}

\begin{proof}
The fact that $\mfd_{\omega_1} = \kappa$ is well known, see for example \cite{cichonunctble}. For $\mfd_{\omega_2}$ it suffices to note that, by the $\aleph_2$-c.c., the forcing is $\aleph_2^{\aleph_2}$-bounding and hence the ground model functions $f:\aleph_2 \to \aleph_2$ form a dominating family.
\end{proof}

We also will use the following result several times.

\begin{lemma}
Assume $\GCH$. Fix $R \in \{\neq^*, \leq^*, \in^*\}$. Let $G\subseteq add(\omega_1, \omega_4)$ be generic over $V$. In $V[G]$ let $\mathbb Q \in V$ be ccc (in $V[G]$) and suppose $\forces_\mathbb Q$ ``$\mfb(R) = \aleph_1$ and $2^{\aleph_0} = \aleph_2$". Let $H \subseteq \mathbb Q$ be generic over $V[G]$. Then in $V[G][H]$ we have that $\mfd(R_{\{\emptyset\}}) = \aleph_4$.
\label{productford}
\end{lemma}

\begin{proof}
First work in $V[G][H]$ and note that the cardinal arithmetic in this model is $2^{\aleph_0} = \aleph_2$ and $2^{\aleph_1} = 2^{\aleph_2} = 2^{\aleph_3} = \aleph_4$ and by Lemma \ref{cohensubsets} $,\mfd_{\aleph_1} = \aleph_4$ and $\mfd_{\aleph_2} = \aleph_3$. It follows that every $\mfd(R_{\{\emptyset\}})$ is either $\aleph_3$ or $\aleph_4$ in light of Corollary \ref{lemmamfd}. Therefore to prove the lemma it suffices to show that $\mfd(R_{\{\emptyset\}}) > \aleph_3$. 

Suppose towards a contradiction that $\{\dot{f}_\alpha \;  |\; \alpha < \omega_3\}$ names a $R_{\{\emptyset\}}$-dominating family in $V[G][H]$. Using the product lemma, plus properness, we can find an intermediate extension containing all of the $\dot{f}_\alpha$'s and $H$. Call this $W\subseteq V[G][H]$. By the ccc of $\mathbb Q$, the remainder forcing $\mathbb R \in W$ is $\omega$-distributive and, in fact, is isomorphic to $add(\omega_1, \omega_4)^V$. Moreover we can assume that $W \models \mfb(R) = \aleph_1$ since this is true in the final model so we can add the unbounded family. Fix an $\aleph_1$-sized $R$-unbounded family $\{x_\alpha \; | \; \alpha < \omega_1\}$. Also, enumerate all of the reals as $\{y_\alpha \; | \; \alpha < \omega_2\}$. Let $\{\dot{g}_\alpha \; | \; \alpha < \omega_2\}$ enumerate $\mathbb R$-names for $\omega_2$ many generic functions $\omega_1 \to \omega_1$. Via a bijection from $\omega_2 \times \omega_1$ to $\omega_2$, we can think of this as one function from $\omega_2$ to $\omega_1$. Let's name this function $\dot{g}$. Consider now a name $\dot{g} '$ for the function $g':\baire \to \baire$ in $V[G][H]$ so that $\dot g'(y_\alpha) = x_\beta$ if and only if $\dot{g}(\alpha) = \beta$. The following claim leads to the contradiction which completes the proof.

\begin{claim}
In $W$ we have $\forces_\mathbb R$ ``$\dot{g}'$ is not $R_{\{\emptyset\}}$-dominated by the $\dot{f}_\alpha$'s.
\end{claim}

\begin{proof}[Proof of Claim]
Fix $\alpha < \omega_3$, let $p \in \mathbb R$ and let $\gamma < \omega_2$ be so that $p$ does not decide $\dot{g}(\gamma)$. Now we can simply extend $p$ to some $q$ which forces $\dot{g} (\gamma) = \beta$ so that $x_\beta$ is not $R$-below $\dot{f}_\alpha(y_\gamma)$ (since $\{x_\alpha \; | \; \alpha < \omega_1\}$ is an unbounded family). But then $q$ forces that $\dot{f}_\alpha$ is not an $R_{\{\emptyset\}}$-bound on $\dot{g}'$ as desired.
\end{proof}

\end{proof}

We now begin with the first model towards proving Theorem \ref{dcardinals}.

\begin{lemma}
Assume $\GCH$. Let $G \subseteq add(\omega_1, \omega_4)$ be generic over $V$. Let $H$ be generic for $\mathbb B_{\omega_2}$ over $V[G]$. Then in $V[G][H]$ we have:
\begin{enumerate}
\item
$\mfd(\neq^*_{\{\emptyset\}}) = \aleph_3$.
\item
$\mfd(\leq^*_{\{\emptyset\}}) = \mfd(\in^*_{\{\emptyset\}}) = \aleph_4$.
\end{enumerate}
\label{randomdcardinal}
\end{lemma}

\begin{proof}
As in the proof of Lemma \ref{productford}, note that the cardinal arithmetic in this model is $2^{\aleph_0} = \aleph_2$ and $2^{\aleph_1} = 2^{\aleph_2} = 2^{\aleph_3} = \aleph_4$ and by Lemma \ref{cohensubsets}, $\mfd_{\aleph_1} = \aleph_4$ and $\mfd_{\aleph_2} = \aleph_3$. It follows that every $\mfd(R_{\{\emptyset\}})$ is either $\aleph_3$ or $\aleph_4$ in light of Corollary \ref{lemmamfd}. From now on work in $V[G]$.

Let us begin by showing that $\mfd(\neq^*_{\{\emptyset\}}) = \aleph_3$. Indeed, it suffices to see that $\mfd(\neq^*_{\{\emptyset\}}) \leq \aleph_3$ since it the lower bound is for free in light of Corollary \ref{lemmamfd}. However, $\mfd(\neq^*_{\{\emptyset\}}) \leq \aleph_3$ actually follows from the first part of Lemma \ref{eventualsequences} given the fact that there is an eventually $\neq^*$-dominating sequence of length $\omega_2$ in the random model.


Now we turn to the proof that $\mfd(\leq^*_{\{\emptyset\}}) = \aleph_4$. Since $\mfd(\leq^*_{\{\emptyset\}}) \leq \mfd(\in^*_{\{\emptyset\}}) \leq \aleph_4$ this will complete the proof of the theorem. This however follows from the fact that $\mfb = \aleph_1$ in the random model alongside an application of Lemma \ref{productford}. Alternatively, we can simply use the fact that in the random model there is a scale of length $\aleph_1$ and apply Lemma \ref{scale}.
\end{proof}

Moving on to the next model we have the following.
\begin{lemma}
Assume $\GCH$. Let $G \subseteq add(\omega_1, \omega_4)$ be generic over $V$. Let $H$ be generic for $\mathbb D_{\omega_2}$,  the finite support iteration of length $\omega_2$ of Hechler forcing, over $V[G]$. Then in $V[G][H]$ we have:
\begin{enumerate}
\item
$\mfd(\neq^*_{\{\emptyset\}}) = \mfd(\leq^*_{\{\emptyset\}}) = \aleph_3$.
\item
 $\mfd(\in^*_{\{\emptyset\}}) = \aleph_4$.
\end{enumerate}
\label{hechlerdcardinal}
\end{lemma}

\begin{proof}
The proof of this theorem is very similar to that of Lemma \ref{randomdcardinal}. As in that proof, in $V[G]$ we have that $2^{\aleph_0} = \aleph_2$ and $2^{\aleph_1} = 2^{\aleph_2} = 2^{\aleph_3} = \aleph_4$ and by Lemma \ref{cohensubsets} $\mfd_{\aleph_1} = \aleph_4$ and $\mfd_{\aleph_2} = \aleph_3$. The Hechler reals added form a scale of size $\aleph_2$ (and there is no such scale of size $\aleph_1$). Again by Lemma \ref{scale} it follows that $\mfd(\leq^*_{\{\emptyset\}}) = \aleph_3$, which finishes the first item.

For the second item we simply apply Lemma \ref{productford} alongside the fact that $\mfb(\in^*)$ is $\aleph_1$ in the Hechler model.
\end{proof}

Now we consider the model where all three cardinals are different.
\begin{lemma}
Assume $\GCH$. Let $G \subseteq add (\omega_2, \omega_5)$ be generic over $V$. Let $H \subseteq add(\omega_1, \omega_6)$ be generic over $V[G]$. In $V[G][H]$ let $\mathbb P$ be the forcing consisting of an $\omega_2$-length iteration of Hechler forcing followed by adding $\omega_3$ many random reals, $\mathbb D_{\omega_2} * \dot{\mathbb B}_{\omega_3}$ and let $K \subseteq \mathbb P$ be generic over $V[G][H]$. In $V[G][H][K]$ we have:
\begin{enumerate}
\item
$\mfd(\in^*_{\{\emptyset\}}) = \aleph_6$
\item
$\mfd(\leq^*_{\{\emptyset\}}) = \aleph_5$
\item
$\mfd(\neq^*_{\{\emptyset\}}) = \aleph_4$
\end{enumerate}
\label{mixdcardinal}
\end{lemma}

\begin{proof}
Modifying the arguments from Lemmas \ref{cohensubsets} and \ref{productford} alongside well known facts about these forcing notions we have that in $V[G][H][K]$ the following all hold.

\begin{enumerate}
\item
$2^{\aleph_0} = \aleph_3$ and for all $\kappa \in [\aleph_1, \aleph_5]$, $2^{\kappa} = \aleph_6$,
\item
$\mfb(\in^*) =\aleph_1 < \mfb = \mfd = \aleph_2 < \mfb(\neq^*) = \aleph_3$,
\item
$\mfd_{\aleph_3} = \aleph_4$, $\mfd_{\aleph_2} = \aleph_5$ and $\mfd_{\aleph_1} = \aleph_6$.
\end{enumerate}

Now, by essentially the same argument as in Lemma \ref{productford} we get that $\mfd (\in^*_{\{\emptyset\}}) = \aleph_6$. To show that $\mfd(\leq^*_{\{\emptyset\}}) = \aleph_5$ we will apply Theorem \ref{mfc=mfb+n}. We need to see that $\mfd^{\aleph_3}_{\aleph_2}$ is $\aleph_5$. To see why this is true, observe that a simply density argument forces that $\mfd^{\aleph_3}_{\aleph_2} = \aleph_5$ in $V[G]$ and since the remainder forcing is $\aleph_2$-c.c., the set $\omega_2^{\omega_3} \cap V[G]$ forms a dominating family in $\omega_2^{\omega_3}$ while no family of smaller cardinality does.

Finally to show that $\mfd(\neq^*_{\{\emptyset\}}) = \aleph_4$ we can use the same  argument as in Lemma \ref{randomdcardinal} (but upping each cardinal by one).
\end{proof}

The final two models for proving Theorem \ref{dcardinals} follow (under $\CH$) from Theorem \ref{CHtheorem2} alongside well known facts about $\mfd_{\aleph_1}$, see for example \cite{CS95}. However, they can also be proved in the $\neg \CH$ context. In the case that all $\mfd(R_{\{\emptyset\}})$ cardinals are greater than $\mfc^+$, we can add $\omega_4$ many Cohen subsets to $\omega_1$ and then $\aleph_2$ many Cohen reals and apply Lemma \ref{productford}, noting that in the Cohen model, $\mfb(\in^*) = \mfb = \mfb(\neq^*) = \aleph_1$.  In the case where all $\mfd(R_{\{\emptyset\}})$ cardinals are less than $2^\mfc$, first force to add $\omega_4$ many Cohen subsets to $\omega_1$ and then force with a finite support product of localization forcing $\mathbb{LOC}$ of length $\omega_2$. The resulting generic for the localization iteration will be a sequence of $\omega_2$-many slaloms $\{s_\alpha \; | \; \alpha < \omega_2\}$ which is eventually $\in^*$-dominating. It follows from the first item in Lemma \ref{eventualsequences} that $\mfd(R_{\{\emptyset\}}) \leq \mfd^\mfc_{\aleph_2} = \aleph_3$ and hence all $\mfd(R_{\{\emptyset\}})$ cardinals are $\aleph_3$. See \cite[Section 3.8]{Switz20} for more on properties of the localization poset.

\section{Conclusion and Open Questions}

The results of the previous section begin to give a more complete picture of the relations $R_\mathcal I$ and their cardinal characteristics. Nevertheless several open questions remain. We conclude by discussing them. The first question concerns the final inequality between the bounding numbers whose negation does not follow from Figures 1, 2 and yet does not hold in any of the models in Section 4. 

\begin{question}
Is $\mfb(\neq^*_\Null) < \mfb(\in^*_\Kb)$ consistent?
\label{conineq}
\end{question}

The issue in providing a model where $\mfb(\neq^*_\Null) < \mfb(\in^*_\Kb)$ is consistent is that for most of our consistent inequalities, the smaller one is shown to be $\aleph_1$ using Lemma \ref{lemmab<cov}. However, by Proposition \ref{neqN} this is not possible in this case. Indeed, if ${\rm add}(\Null) < \mfd$ then by Lemma \ref{lemmab<cov} we have that $\mfb(\in^*_\Kb) = {\rm add}(\Null)$ so in such a model $\mfb(\in^*_\Kb) \leq \mfb(\neq^*_\Null)$. It follows that if  $\mfb(\neq^*_\Null) < \mfb(\in^*_\Kb)$ is consistent we must have that ${\rm add}(\Null) = \mfd < \mfc$, however in all the standard models of this inequality we have computed, namely the random and Sacks models, the inequality $\mfb(\neq^*_\Null) < \mfb(\in^*_\Kb)$ does not hold.

The next question we have also concerns the bounding numbers. Recall that the final $\mfb(R_\mathcal I)$ number not computed is $\mfb(\in^*_\Me)$ in the Laver model.

\begin{question}
What is the value of $\mfb(\in^*_\Me)$ in the Laver model?
\end{question}

As noted in the section on the Laver model, this cardinal must be either $\aleph_2$ or $\aleph_3$. A related question is the following.
\begin{question}
Is there an $(\aleph_2, \aleph_1)$-Rothberger family for $\Me$ in the Laver model?
\end{question}
Note that a positive answer to this question implies that $\mfb(\in^*_\Me) = \aleph_3$ in the Laver model by Lemma \ref{coveringfamilylemma}.

We can also ask about Rothberger families more generally. As noted after the proof of Lemma \ref{coveringfamilylemma}, every computation of the form $\mfb(R_\mathcal I) > \mfc$ can be thought of as factoring through this lemma. It is therefore worth asking if this is necessary.

\begin{question}
Given $R$ and $\mathcal I$, assume $\mfb(R_\mathcal I) > \mfc$. Does it follow that there exists a $(\mfc, \mfb(R))$- Rothberger family for $\mathcal I$ as in Lemma \ref{coveringfamilylemma}?
\end{question}

Moving on the dominating numbers, the main open question, as discussed in Section 5 is the relation between the cardinals of the form $\mfd(R_{\{\emptyset\}})$ and the numbers $\mfd^\lambda_\kappa$.

\begin{question}
Is every $\mfd(R_{\{\emptyset\}})$ cardinal equal to one of the form $\mfd_\kappa^\mfc$? In particular, is $\mfd(R_{\{\emptyset\}})$ $\ZFC$-provably equal to $\mfd^\mfc_{\mfb(R)}$?
\label{dkappalambdaq}
\end{question}

One can ask a more general question along these lines. Namely if the cardinal characteristics on spaces of the form $\kappa^\lambda$ with $\kappa \leq \lambda \leq \mfc$ determine any of the bounding or dominating numbers for the $R_\mathcal I$ relations.

\begin{question}
Are any of the cardinal characteristics discussed determined by the values of the cardinals in the Cicho\'{n} diagram and its higher analogues on $\kappa^\lambda$ for $\kappa \leq \lambda \leq \mfc$?
\end{question}

Put another way, is there either a $\mfb(R_\mathcal I)$ or a $\mfd(R_{\{\emptyset\}})$ cardinal whose value is completely determined by the values in the Cicho\'{n} diagrams on $\kappa^\kappa$ for $\kappa \in [\omega, \mfc]$ and the numbers $\mfb^\lambda_\kappa$ and $\mfd^\lambda_\kappa$ for $\kappa \leq \lambda \leq \mfc$? We know by the computation of $\mfb(\in^*_\Me)$ in the $\lambda$-dual random model that this is not true in general, but could it be true for some of the cardinals? Note that a positive answer to Question \ref{dkappalambdaq} implies a positive answer to this question in the case of the dominating numbers.

Finally we ask about the relationship between the $\mfb(R_\mathcal I)$ and the $\mfd(R_\mathcal I)$ cardinals. These relations were not considered in this paper at all. 

\begin{question}
Can some $\mfb (R_\mathcal I)$ be consistently strictly larger than some $\mfd (R_\mathcal J)$?
\end{question}
The only relation $R$ for which is is not explicitly ruled out by Figures 1 and 2 is the relation $\neq^*$, so really the question is whether there is an ideal $\mathcal I \in \{\Null, \Me, \Kb\}$ so that $\mfb(\neq^*_\mathcal I) > \mfd(\neq^*_{\{\emptyset\}})$ is consistent.

\end{document}